\documentclass[hidelinks,onefignum,onetabnum]{siamart220329}

\usepackage{lipsum}
\usepackage{amsfonts}
\usepackage{graphicx}
\usepackage{epstopdf}
\usepackage{algorithmic}
\ifpdf
  \DeclareGraphicsExtensions{.eps,.pdf,.png,.jpg}
\else
  \DeclareGraphicsExtensions{.eps}
\fi

\newsiamremark{remark}{Remark}
\newsiamremark{hypothesis}{Hypothesis}
\crefname{hypothesis}{Hypothesis}{Hypotheses}
\newsiamthm{claim}{Claim}

\headers{Global Polynomial Level Sets}{S. K. Thekke Veettil, G. Zavalani, U. H. Acosta, I. F. Sbalzarini, and M. Hecht}

\title{Global Polynomial Level Sets for Numerical Differential Geometry of Smooth Closed Surfaces\thanks{
\funding{This work was partially funded by the Center of Advanced Systems Understanding (CASUS), financed by Germany's Federal Ministry of Education and Research (BMBF) and by the Saxon Ministry for Science, Culture and Tourism (SMWK) with tax funds on the basis of the budget approved by the Saxon State Parliament.
}}}

\author{Sachin Krishnan Thekke Veettil\thanks{Technische Universit\"{a}t Dresden,
Faculty of Computer Science, Dresden, Germany.\\
\indent \indent Max Planck Institute of Molecular Cell Biology and Genetics, Dresden, Germany.\\
\indent \indent Center for Systems Biology Dresden, Dresden Germany
}
\and Gentian Zavalani\thanks{Center for Advanced Systems Understanding (CASUS), G\"{o}rlitz, Germany.}
\and Uwe Hernandez Acosta\footnotemark[3]
\and Ivo F. Sbalzarini\footnotemark[2] \thanks{Center for Scalable Data Analytics and Artificial Intelligence ScaDS.AI, Dresden, Germany.}
\and Michael Hecht\footnotemark[3] \thanks{Corresponding author. Email: m.hecht@hzdr.de}}

\usepackage{amsopn}

\ifpdf
\hypersetup{
  pdftitle={Global Polynomial Level Sets for Numerical Differential Geometry of Smooth Closed Surfaces},
  pdfauthor={S. K. Thekke Veettil, G. Zavalani, U. H. Acosta, I. F. Sbalzarini, and M. Hecht}
}
\fi

\usepackage{amsfonts,amsmath}
\usepackage{mathptmx}
\usepackage{protosem}
\usepackage[numbers,sort&compress]{natbib}

\newcommand{\R}{\mathbb{R}}
\newcommand{\C}{\mathbb{C}}
\newcommand{\N}{\mathbb{N}}

\newcommand{\Mc}{\mathcal{M}}
\newcommand{\Lc}{\mathcal{L}}
\newcommand{\Pc}{\mathcal{P}}

\newcommand{\dist}{\mathrm{dist}}

\newcommand{\Cheb}{\mathrm{Cheb}}
\newcommand{\Oc}{\mathcal{O}}
\newcommand{\ee}{\varepsilon}
\newcommand{\lo}{\longrightarrow}
\newcommand{\li}{\left}
\newcommand{\re}{\right}

\newcommand{\rank}{\mathrm{rank}}

\newcommand{\p}{\partial}

\newcommand{\GR}{\mathrm{Gr}}

\newtheorem{experiment}{\sc Experiment}
{\bf}{\it}

\begin{document}

\maketitle

% REQUIRED
\begin{abstract}
We present a computational scheme that derives a global polynomial level set parametrisation for smooth closed surfaces from a regular
surface-point set and prove its uniqueness. This enables us to approximate a broad class of smooth surfaces by affine algebraic varieties. From such a
global polynomial level set parametrisation, differential-geometric quantities like mean and Gauss curvature can be efficiently and accurately computed. Even 4$^{\text{th}}$-order terms such as the Laplacian of mean curvature are approximates with high precision. The accuracy performance results in a gain of computational efficiency, significantly reducing the number of surface points required compared to classic alternatives that rely on surface meshes or embedding grids. We mathematically derive and empirically demonstrate the strengths and the limitations of the present approach, suggesting it to be applicable to a large number of computational tasks in numerical differential geometry.
\end{abstract}

% REQUIRED

% REQUIRED

\section{Introduction}

Classic differential geometry of closed two-dimensional surfaces $S \subseteq \R^3$, $\partial S = \emptyset$, goes back to Carl Friedrich Gauss \cite{gauss2005general,kuhnel2005differential}, Bernhard Riemann \cite{berger2000riemannian,jost2008riemannian}, and others.
Numerically computing or approximating such surfaces'
main geometric quantities, like \emph{Gauss and mean curvature}, is of fundamental importance
across scientific disciplines such as \emph{biophysics} \cite{mietke2019minimal}, mechanics \cite{schwartz1998simulation,sander2016numerical},
medical imaging \cite{khan2018single}, sociology \cite{gomes20143d},
and computer graphics \cite{Calakli2011SSDSS,taubin2012smooth}.
High accuracy of these approximations is key in many applications, including dynamic surface models where deformations are governed
 by  the \emph{intrinsic Laplacian of curvature}~\cite{sethian1999motion}, \emph{surface diffusion}~\cite{smereka2003,greer2006fourth}, and the \emph{dynamics of cell membranes and vesicles}~\cite{seifert1997configurations}. Such models that require accurate numerical computation of 4$^\text{th}$-order differential terms, such as the Laplacian of mean curvature, present a challenge to available numerical methods.

We here address this challenge by combining algebraic geometry with classic numerical analysis in order to
 formulate a mathematical theory that enables us to
approximate smooth closed surfaces $S \approx Q_S^{-1}(0)$ by algebraic varieties (i.e., hypersurfaces) with global polynomial level set (GPLS) parametrisation $Q_S$.
As we demonstrate here, the GPLS can be numerically computed in an efficient way for a large class of surfaces. The GPLS can subsequently be used to compute geometric quantities with high precision, enabling efficient approximation of higher-order quantities.

% % The outline is not required, but we show an example here.
% The paper is organized as follows. Our main results are in
% \cref{sec:main}, our new algorithm is in \cref{sec:alg}, experimental
% results are in \cref{sec:experiments}, and the conclusions follow in
% \cref{sec:conclusions}.

\section{Related work}
The importance of the present computational challenge is manifest in the large number of previous works.
Consequently, an exhaustive overview of the literature cannot be given here.
Instead, we restrict ourselves to mentioning those contributions that directly relate to or inspired our work.
%Methods addressing the computation of differential geometric quantities (curvatures, Laplacian of curvature, etc.) can be divided into explicit and implicit methods:
This includes methods where tracer points $P \subseteq S$ are placed on the surface in order to approximate $S$ by \emph{local interpolation} over finite neighbourhoods. Well-established interpolation methods include approaches based on \emph{B-splines} \cite{deBoor1971,deBoor1978, barnhill1974}, \emph{Galerkin mesh-based} \cite{ruppert1995delaunay,shewchuk2002delaunay} \emph{finite element methods} \cite{fletcher1984computational,Fletcher1984,deckelnick2005,dziuk2007,dziuk2013,sander2020},
% , and Poisson reconstructions\cite{Estellers2015RobustPS,Kazhdan2013ScreenedPS},
and triangulated surface methods \cite{Kroll1992}.

Alternatively, surfaces can be approximated by a discrete or discretised implicit representation. This includes
\emph{level set methods}
 \cite{osher1988fronts,sethian1996theory,sethian1997tracking,sethian1999level,osher2005level}, \emph{local kernel (radial basis function) parametrisations}
 \cite{ztireli2009FeaturePP, Carr2001ReconstructionAR,Casciola2006ShapePS,Huang2009ConsolidationOU}, \emph{closest point methods}
 \cite{ruuth2008simple,macdonald2008level,macdonald2011solving,macdonald2013simple}, and \emph{phase field methods}~\cite{ratz2006diffuse,ratz2006pde,provatas2011phase}.

All of these approaches have in common that they approximate the surface with discrete points, meshes, or grids. Then, differential geometric quantities are computed on these discrete surface approximations using numerical methods. This introduces a second approximation, namely of the (surface) differential operators by their discretizations on the discrete surface approximation. This often prevents reaching the levels of accuracy required to compute higher-order geometric quantities.

The present approach avoids the second approximation by representing the surface globally as an algebraic variety. This leads to a surface representation that is continuous (even smooth) and defined everywhere, albeit supported on a finite set of discrete surface points. Doing so in a proper polynomial basis, we can compute any differential quantities to machine precision without introducing another approximation. While it has long been known that polynomial surface approximations have some desirable properties, and some methods have computed them piecewise,
see e.g., \cite{piegl1996nurbs,gershenfeld1999nature},
we here exploit a recent advancement in polynomial interpolation \cite{IEEE} that allows us to compute such representations globally.

\section{Main results}
\label{sec:main}

The GPLS method presented here determines a multivariate polynomial $Q_S(x)$, $x \in \R^3$, from points on a surface $P \subseteq S$ such that the surface is approximated by the zero-level set (zero contour) $S \approx Q_S^{-1}(0)$ of the polynomial. We do so for two classes of closed surfaces $S \subseteq \R^3$, $\p S=\emptyset$:
\begin{enumerate}
  \item[C1)] algebraic varieties $M$ of low degree: in this case the GPLS approach amounts to a \emph{mesh-free particle method} that only relies on \emph{a regular
  point set} $P \subseteq M$ sampled on the surface and does not require any surrounding (narrow band) grid or mesh.
  \item[C2)] non-algebraic surfaces $S$: in this case the GPLS approach amounts to fitting a global polynomial approximation of a given discretised \emph{(relaxed) signed distance function} and yields an alternative classic level set redistancing methods \cite{sussman1998improved,sussman1999efficient}.
\end{enumerate}

In both cases, $Q_M^{-1}(0)$ approximates (up to the interpolation or fitting error) the original surface by a uniquely determined algebraic variety, see Theorem~\ref{theorem:Dual}$(iv,v,vi)$. This means that regardless of whether the original surface was algebraic or not, the GPLS approximation of it always is. It also means that the exact same (unique) surface approximation is obtained from a given set of surface points $P$ regardless of the maximum polynomial degree chosen. The unique polynomial approximation of the surface can then be used to accurately compute \emph{mean and Gauss curvature} as well as their derivatives, e.g., the \emph{Laplacian of mean curvature}, as we demonstrate in the numerical experiments of Section \ref{sec:Num}.

\subsection{Notation}

Let $m,n \in \N$, $p>0$. Throughout this article, $\Omega=[-1,1]^m$ denotes the $m$-dimensional \emph{standard hypercube} and $C^0(\Omega,\R)$ the
\emph{Banach space}
of continuous functions $f : \Omega \lo \R$ with norm $\|f\|_{C^0(\Omega)} = \sup_{x \in \Omega}|f(x)|$. We denote by  $e_1=(1,0,\dots,0),\, \dots,\, e_m = (0,\dots,0,1) \in \R^m$ the standard basis, by $\|\cdot\|_p$ the $l_p$-norm on $\R^m$, and by $\|M\|_p$ the $l_p$-norm
of a matrix $M\in \R^{m\times m}$. Further, $A_{m,n,p} \subseteq \N^m$ denotes all multi-indices $\alpha =(\alpha_1,\dots,\alpha_m)\in \N^m$ with $\|\alpha\|_p \leq n$.
We order $A_{m,n,p}$ with respect to the lexicographical order $\preceq$ on $\N^m$ going from the last entry to the first, e.g.,
$(5,3,1)\preceq (1,0,3) \preceq (1,1,3)$.
We call $A$
\emph{downward closed} if and only if there is no $\beta = (b_1,\dots,b_m) \in \N^m \setminus A$
with $b_i \leq a_i$, $ \forall \,  i=1,\dots,m$ for some $\alpha = (a_1,\dots,a_m) \in A$ \cite{cohen3}.

The sets $A_{m,n,p}$ are downward closed for all $m,n\in \N$, $p>0$, and induce a generalised notion of \emph{polynomial $l_p$-degree} as follows:
We consider the \emph{real polynomial ring} $\R[x_1,\dots,x_m]$  in $m$ variables and denote by $\Pi_m$ the $\R$-\emph{vector space of all real polynomials} in $m$ variables.
For $A\subseteq \N^m$, $\Pi_{A} \subseteq \Pi_m$ denotes the \emph{polynomial subspace} $\Pi_A = \mathrm{span}\{x^\alpha\}_{\alpha \in A}$ spanned by the (unless further specified) \emph{canonical} (monomial) \emph{basis}. Choosing $A =A_{m,n,p}$ yields the spaces
$\Pi_{A_{m,n,p}}$.
The particular cases of \emph{total degree} $A = A_{m,n,1}$, \emph{Euclidean degree} $A= A_{m,n,2}$, and \emph{maximum degree} $A= A_{m,n,\infty}$
will play an important role for the polynomial approximation quality. As noticed by \cite{Lloyd2}, the sizes of these sets scale polynomially, sub-exponentially, and exponentially with dimension, respectively:
\begin{equation}\label{eq:size}
  |A_{m,n,1}| = \binom{m+n}{n} \in \Oc(m^n)\,, \,\,\, |A_{m,n,2}|\approx \frac{(n+1)^m }{\sqrt{\pi m}} \li(\frac{\pi \mathrm{e}}{2m}\re)^{m/2} \in o(n^m)\,, \,\,\, |A_{m,n,\infty}| = (n+1)^m\,.
\end{equation}
Given linear ordered sets $A \subseteq \N^m$, $B \subseteq \N^n$, we slightly abuse notation by writing matrices $R_{A,B}\in \R^{|A|\times |B|}$ as
\begin{equation}
 R_{A,B} = (r_{\alpha,\beta})_{\alpha\in A, \beta \in B} \in \R^{|A|\times|B|}\,,
\end{equation}
where $r_{\alpha,\beta} \in \R$ is the $\alpha$-th, $\beta$-th entry of $R_{A,B}$.
%IFS: I'm not getting this notation. Aren't \alpha and \beta multi-indices?
%MH: better ?
%IFS: not really. I am confused by the subscripts of r being multi-indices and not scalars, as would be required to denote a matrix element.... or what is r ??
%MH: now better ?
Finally, we use the standard \emph{Landau symbols} $f \in \Oc(g)  \Longleftrightarrow \lim \sup_{x\rightarrow \infty} \frac{|f(x)|}{|g(x)|} \leq \infty$, $f \in \omicron(g) \Longleftrightarrow \lim_{x\rightarrow \infty} \frac{|f(x)|}{|g(x)|} =0$.

\section{Unisolvent nodes and multivariate interpolation}
\label{sec:UN}

We briefly summarise here the essential concepts from \cite{cohen2,cohen3,PIP1,PIP2,MIP,IEEE} on which our approach rests, in particular the notion of unisolvence with respect to generalised polynomial degree.

\subsection{The notion of unisolvence}
For a downward closed multi-index set $A \subseteq \N^m$, $m\in N$, and the induced polynomial space $\Pi_A$, a set of nodes $P \subseteq \Omega$ is called \emph{unisolvent with respect to $\Pi_A$}
if and only if there exists no hypersurface $H = Q^{-1}(0)$
generated by a polynomial $0\not =Q \in \Pi_A$ with $P \subseteq H$.
The opposite notion of non-unisolvent nodes $P$ allows us to derive global polynomial hypersurfaces $Q_S^{-1}(0) \supseteq P$ that contain a given (regular) point set $P \subseteq S$ and approximate the initial surface $S$, see Section~\ref{sec:Dual}.
The following concepts are useful in the derivation:

\begin{definition}[1$^\text{st}$  and 2$^\text{nd}$  essential assumptions]
\label{def:EA}
Let $m \in \N$, $A\subseteq \N^m$ be a downward closed set of multi-indices, and $\Pi_A \subseteq \Pi_m$ the polynomial sub-space induced by $A$.
We consider the generating nodes given by the grid
\begin{equation}\label{GP}
\mathrm{GP}= \oplus_{i=1}^m P_i\,, \quad P_i =\{p_{0,i},\dots,p_{n_i,i}\} \subseteq \R \,, \,\,\,n_i=\max_{\alpha \in A}(\alpha_i)\,,
\end{equation}
\begin{equation}\label{eq:UN}
   P_A = \li\{ (p_{\alpha_1,1}\,, \dots \,, p_{\alpha_m,m} ) : \alpha \in A\re\} \,.
 \end{equation}
\begin{enumerate}
  \item[A1)] If the $P_i \subseteq [-1,1]$ are arbitrary distinct points then the node set $P_A$
  is said to satisfy the \emph{1$^\text{st}$ essential assumption}.
  \item[A2)] We say that the \emph{2$^\text{nd}$ essential assumption} holds if in addition the $P_i$ are chosen as the Chebyshev-Lobatto nodes that, in addition, are Leja-ordered \cite{leja}, i.e,
  \begin{equation*}
  P_i =\{p_0,\dots,p_n\} = \pm \Cheb_n = \li\{ \cos\Big(\frac{k\pi}{n}\Big) : 0 \leq k \leq n\re\}
  \end{equation*}
  and the following holds
  \begin{equation}\label{LEJA}
   |p_0| = \max_{p \in P}|p|\,, \quad \prod_{i=0}^{j-1}|p_j-p_i| = \max_{j\leq k\leq m} \prod_{i=0}^{j-1}|p_k-p_i|\,,\quad 1 \leq j \leq n\,.
  \end{equation}
\end{enumerate}
\end{definition}
Note that $\{p_0,p_1\} = \{-1,1\}$ for all Leja-ordered Chebyshev-Lobatto nodes $\Cheb_n$ with $n\geq 1$.
Points $P_A$ that fulfil the \emph{1$^\text{st}$ essential assumption} form non-tensorial (non-symmetric) grids and are unisolvent with respect to $\Pi_A$ \cite{cohen2,cohen3,PIP1,PIP2,MIP,IEEE}.
This allows generalising classic interpolation approaches to higher dimensions.

\subsection{Multivariate Newton and Lagrange interpolation}

Given unisolvent nodes $P_A$, multivariate generalisations of the classic 1D Newton and Lagrange interpolation schemes can be derived, see e.g., \cite{LIP,IEEE}:
\begin{definition}[Multivariate Lagrange polynomials] Let $m \in \N$, $A\subseteq \N^m$ be a downward closed set of multi-indices, $P_A\subseteq \Omega$ be a set of unisolvent nodes satisfying $(A1)$ from Definition~\ref{def:EA}, and $\Pi_A =\mathrm{span}\{x^\alpha\}_{\alpha \in A}\subseteq \Pi_m$ be the corresponding canonical polynomial space.
We define the \emph{multivariate Lagrange polynomials}  $L_\alpha \in \Pi_A$ by
\begin{equation}
  L_\alpha(p_\beta) = \delta_{\alpha,\beta}\,,
\end{equation}
where $\delta_{\cdot,\cdot}$ is the Kronecker delta.
\end{definition}

Since the $|A|$-many Lagrange polynomials are linearly independent functions, and \linebreak
$\dim \Pi_A=|A|$, the Lagrange polynomials are a basis of $\Pi_A$.
Consequently, any function $f : \Omega \lo \R$ possesses a unique interpolant $Q_{f,A} \in \Pi_A$ with $Q_{f,A}(p_{\alpha}) = f(p_{\alpha})$ $\forall \alpha \in A$, given by
\begin{equation}
  Q_{f,A}(x) = \sum_{\alpha \in A}f(p_{\alpha})L_{\alpha}(x)\,, \quad x \in \R^m\,.
\end{equation}
However, this does not allow for efficient  evaluation of $Q_{f,A}$ at an argument $x_0 \not \in P_A \subseteq \R^m$.
For that, the Newton basis of $\Pi_A$ is better suited:

\begin{definition}[Multivariate Newton polynomials] Let $m \in \N$, $A\subseteq \N^m$ be a downward closed set of multi-indices, $P_A\subseteq \Omega$ be a set of unisolvent nodes satisfying $(A1)$ from Definition~\ref{def:EA}. Then,
the \emph{multivariate Newton polynomials} are given by
 \begin{equation}\label{Newt}
  N_\alpha(x) = \prod_{i=1}^m\prod_{j=0}^{\alpha_i-1}(x_i-p_{j,i}) \,, \quad \alpha \in A\,,
 \end{equation}
 where $p_{j,i} \in P_i$ from Eq.~\eqref{GP}.
 \label{def:LagP}
\end{definition}

In dimension $m=1$, both of these definitions reduce to the classic definitions of Lagrange and Newton polynomials, see e.g.,\cite{gautschi,Stoer,Lloyd}. In arbitrary dimensions, efficient algorithms exist for computing the interpolant in Newton form as well as for its evaluation at any argument $x_0 \in\R^m$ and its differentiation, see \cite{minterpy}.

%We will link back to this algorithmic perspective in Section~\ref{sec:Num}. Prior, we provide the mathematical theory that enables modeling of 2D surfaces as polynomial hypersurfaces, relying on the dual notion of unisolvence.

\section{The dual notion of unisolvence}\label{sec:Dual}

For the purpose of surface approximation, we use the dual notion of unisolvence.
Rather than asking for nodes $P \subseteq \R^m$, $m \in \N$, that are unisolvent with respect to a given polynomial space $\Pi$, already \cite{deBoor2,deBoor}
asked the dual question of finding a polynomial space $\Pi$ with respect to which a given set of points $P \subseteq \Omega$ is unisolvent. We here formulate this question in a generalised way.

\subsection{Unisolvent polynomial spaces}
In order to formulate the dual notion of unisolvence,
it is useful to consider the \emph{Grassmann manifold} $\GR(k,X)$, i.e.,
the smooth manifold that consists of all $k$-dimensional subspaces of
the vector space $X$ \cite{milnor}. In particular, $\GR(1,\R^m) = \mathbb{RP}^{m-1}$ and $\GR(1,\C^m) = \mathbb{CP}^{m-1}$ are the real and complex \emph{projective spaces}, respectively~\cite{dieudonne1971elements,hatcher_vec}.
Using this notion, we state:
\begin{theorem} \label{theorem:Gamma} Let $m,k \in \N$, $\Pc_k = \li\{ P \subseteq \R^m : |P| =k\re\}$ be the set of all subsets of $\R^m$ of cardinality $k$, and $X = \Pi_{m,k-1,\infty}$ the space of all polynomials
with $l_\infty$-degree at most $k-1$. Then, there is one and only one polynomial subspace $\Pi_P \subseteq X$ of dimension $\dim \Pi_P = k$ such that $P$ is unisolvent with respect $\Pi_P$.
In particular, the map
\begin{equation}\label{Gamma}
  \Gamma_k : \Pc_k \lo \GR(k,X)\,,   \quad P \in  \Pc_k  \,\,\,\text {is unisolvent with respect to} \,\,\, \Gamma_k(P)\,,   \quad X= \Pi_{m,k-1,\infty} ,
\end{equation}
is well-defined and smooth.
\end{theorem}

For $m=1$ we have $X= \Pi_{m,k-1,\infty} = \Pi_{1,k-1,1}$ and $\GR(k,X) = X$. Since $k$ distinct nodes are unisolvent in dimension 1 with respect to $X$, Theorem \ref{theorem:Gamma} becomes trivial, and $\Gamma_k(P)\equiv X$ is constant in that case.

\begin{proof} According to $(A1)$ from Definition~\ref{def:EA}, we choose unisolvent nodes  $P_{A_{m,k-1,\infty}}$ with respect to $X =\Pi_{m,k-1,\infty}$, $\dim X = k^m$, and
denote by $L_\alpha \in X$, $\alpha \in A_{m,k-1,\infty}$, the corresponding Lagrange basis. We fix an ordering $P =\{p_1,\dots,p_k\}$ and consider the corresponding Vandermonde matrix
\begin{equation}\label{eq:R}
  R_{A,P}=(r_{i,\alpha}) \in \R^{|P| \times k^m}\,,\quad r_{i,\alpha} = L_\alpha (p_i)\,, \quad i=1,\dots,|P|\,,  \alpha \in A_{m,k-1,\infty}\,.
\end{equation}
Let $\mu =\rank(R_{A,P})\leq |P| =k$ be the rank of $R$ and $D = \mathrm{diag}(\overbrace{1,\dots,1}^{\mu},\overbrace{0,\dots,0}^{k-\mu}) \in \R^{k \times k}$ the diagonal matrix with the first $\mu$ entries equal to 1 and all others equal to 0.
Let further $C\in \R^{k^m \times k}$ be a solution of $RC = D$ and $C_i =(c_{\alpha,i}) \in \R^{k^m}$ be the $i$-th column of $C$. Then
$$ \Lc_i(x) = \sum_{\alpha \in A_{m,k-1,\infty}} c_{\alpha,i} L_{\alpha}(x) \in X\,, \quad i =1 ,\dots,\mu$$
yields Lagrange polynomials with $\Lc_i(p_j) = \delta_{i,j}$, $1 \leq j \leq \mu$, where $\delta_{\cdot,\cdot}$ denotes the Kronecker delta.
% while ,the $\{L_\alpha\}_{\alpha \in A_{m,k-1,\infty}}$ are a basis of $\Pi_{A_{m,k-1,\infty}}$ \cite{MIP}, the $\Lc_i$ are uniquely determined and linearly independent.
Thus, the $\Lc_i$ are linearly independent.

We argue by contradiction to show that $\mu =k$. Indeed,
if $\mu <k$ then there is no polynomial $Q \in \Pi_{A_{m,k-1,\infty}}$ with $Q(p_j) = 0$ $\forall j =1 ,\dots, \mu$ and  $Q(p_{l}) =1$ for $l >\mu$.
We denote by $p_j=(p_{j,i})_{i=1,\ldots,m}\in P \subseteq \R^m$ the coordinates of the $p_j$, $1 \leq j \leq k$. Then, there is a sequence $p_{j,i_j}$, $1 \leq i_j \leq m$, of coordinate entries such that
$p_{l,i_l}\neq p_{j,i_j}$ for all $1\leq j\leq \mu$. Consequently, setting
$$ Q(x)=  \frac{\prod_{j=1}^\mu (x_{i_j} - p_{j,i_j}) }{\prod_{j=1}^\mu (p_{l,i_l} - p_{j,i_j})} \in \Pi_{A_{m,\mu,\infty}}\subseteq \Pi_{m,k-1,\infty}\,, \quad  l > \mu$$
provides such a polynomial, contradicting $\mu < k$.
Thus, setting $\mathfrak{m}(P) := \{Q \in X : Q(p) = 0 \, \,\forall p \in P\}$ implies that
$$ \Gamma_k(P):= X/\mathfrak{m}(P) \cong \mathrm{span} (\Lc_i)_{i=1,\dots,k}$$
% $\Pi_P:= \mathrm{span} (\Lc_i)_{i=1,\dots,k} \subseteq X=\Pi_{A_{m,k-1,\infty}}$
is the uniquely determined polynomial subspace for which $P$ is unisolvent. Since $\GR(k,X)$ is a smooth manifold \cite{milnor}, and the $\Lc_i$ depend smoothly on $P$, this shows that
$\Gamma_k : \Pc_k \lo \GR(k,X)$ is a well-defined smooth map.
\end{proof}

Theorem~\ref{theorem:Gamma} guarantees that
the algebraic variety $M = Q_M^{-1}(0)$ we derive as GPLS from given points $P \subseteq S$ is uniquely determined.

\section{Algebraic varieties and polynomial hypersurfaces}

Formulating the practical consequences of
Theorem \ref{theorem:Gamma} relies on concepts from algebraic geometry.
An excellent overview of these topics is given by \cite{hatcher}. We start by stating:

\begin{theorem}\label{theorem:Dual} Let $m \in \N$, $A\subseteq \N^m$ be a downward closed set of multi-indices, $P_A\subseteq \Omega$ be a set of unisolvent nodes satisfying $(A1)$ from Definition~\ref{def:EA}. Denote by $\{L_{\alpha}\}_{\alpha \in A} \subseteq \Pi_A$ the Lagrange basis with respect to $\Pi_A$ and $P_A$. Let further $\Gamma_k$ be as in Theorem \ref{theorem:Gamma}.
Given any set of points $P \subseteq \R^m$, the following holds:
\begin{enumerate}
 \item[i)] There is  a set $P_0 \subseteq P$ of maximum cardinality $k=|P_0|$, which can be determined in $\Oc(|A|^3)$ operations, such that $\Gamma_k(P_0) \subseteq \Pi_A$.
 \item[ii)] A Lagrange basis $\{\Lc_1,\dots,\Lc_{k}\} \subseteq \Gamma_k(P_0)$  with $\Lc_i(p_j) = \delta_{i,j}$, $1\leq (i,j) \leq k$, $p_{j} \in P_0$, and a basis
 $\{\Mc_1,\ldots,\Mc_{|A|-k})\} \subseteq \Pi_{A}/\Gamma_k(P_0)$ of the quotient space can be computed in $\Oc(|A|^3)$ operations.
 \item[iii)]
 Consider the affine algebraic variety $M = Q_M^{-1}(0)$ given by the polynomial hypersurface
 \begin{equation}\label{PiM}
   Q_M(x) = \sum_{i=1}^k \Lc_{i}(x)-1\,.
 \end{equation}
Further, let $\mathfrak{m}_M = \li\{ Q \in \Pi_A : Q(x) = 0\, \forall x \in M\re\}$ be the vector space of polynomials identically vanishing on $M$, and
 $\Pi_M = \{Q_{|M} : Q \in \Pi_A\}$ be the vector space of polynomials restricted to $M$.
 Then
 \begin{equation}\label{quotient}
      \Pi_M \cong \Pi_A/\mathfrak{m}_M \cong \Gamma_k(P_0)\,.
 \end{equation}
\item[iv)] If $P \subseteq M =Q_M^{-1}(0)$, with $Q_M$ from Eq.~\eqref{PiM}, then $M$ and $Q_M$ are uniquely determined up to adding polynomials from $\mathfrak{m}_M$, i.e, for any other maximal set $P_0 \neq P_0' \subseteq P$, $|P_0|=|P_0'|=k$, with $\Gamma_k(P_0') \subseteq \Pi_A$ there holds
\begin{equation}\label{eq:QM}
\Gamma_k(P_0') = \Gamma_k(P_0) \quad \text{and}\quad Q_{M}(x) - Q_{M'}(x) = \sum_{i=1}^k \Lc_{i}(x) -  \sum_{i=1}^k \Lc_{i}'(x)  \in  \mathfrak{m}_M\,,
\end{equation}
where $\{\Lc_{i}'\}_{i=1,\ldots,k}$ denotes the Lagrange basis from $ii)$ with respect to $P_0'$, and $\mathfrak{m}_M$ is as in $iii)$.
\item[v)] Let $A_1 \subseteq A_2 \subseteq \N^m$ be two choices of multi-index sets and $P_{A_1},P_{A_2}$ the corresponding unisolvent nodes fulfilling the assumptions of the theorem
such that $P \subseteq M_1 =Q_{M_1}^{-1}(0)$ and $P \subseteq M_2 =Q_{M_2}^{-1}(0)$ holds for the corresponding level sets. Then, the two algebraic varieties are identical
$M_1 =M_2$.
\item[vi)] Let $P_1$ and $P_2$ be two sets of points and $P_1\subseteq M_1=Q_{M_1}^{-1}(0)$, $P_2\subseteq M_2=Q_{M_2}^{-1}(0)$, $Q_{M_1} \in \Pi_{A_1}$, $Q_{M_2} \in \Pi_{A_1}$,  the (due to $(v)$ uniquely determined) corresponding algebraic varieties. If  $P_1 \cup P_2 \subseteq M_1\cap M_2$, then $M_1= M_2$ are identical.
\end{enumerate}
\end{theorem}

\begin{proof} If $P$ is unisolvent with respect to $\Pi_A$, the statement is trivial. For non-unisolvent $P$ all statements follow from the following observation: We order the nodes  $P = \{p_1, \dots , p_{|P|}\}$ and
consider the corresponding Vandermonde matrix, as in Eq.~\eqref{eq:R},
\begin{equation}\label{eq:RA}
  R_{A,P} =(r_{i,\alpha}) \in \R^{|P| \times |A|}\,, \quad  r_{i,\alpha} = L_\alpha (p_i)\,, \quad 1 \leq i \leq |P|\,, \alpha \in A\,.
\end{equation}
By using \emph{Gaussian elimination with full pivoting} (GEFP), see, e.g.,\cite{Lloyd_Num}, we can find a $LU$-decomposition of $R_{A,P}$. That is, there are \emph{permutation matrices} $W_1\in \R^{|P|\times |P|}$,\, $W_2\in \R^{|A|\times |A|}$,
a \emph{unitary lower triangular matrix} $L\in \R^{|P|\times |A|}$, and an \emph{upper triangular matrix}  $U \in \R^{|A|\times |A|}$ such that
\begin{equation}\label{RLU}  W_1 R_{A,P}W_2 =L U
\quad \text{with}\quad  U= \li(\begin{array}{cc}
                                U_1 & U_2 \\
                                0 & 0
                               \end{array}\re)  \in \R^{l\times l},\,  U_1\in \R^{k \times k},\, U_2\in \R^{k \times |A|-k}
\end{equation}
and the $k$ diagonal entries of $U_1$ do not vanish.
Consequently,  $\rank(R_{A,P}) = \rank(U) =\rank(U_1) =k \in \N$.  Let
$P_0 = \{p_1,\dots,p_k\}$ be the first $k$ nodes of the node set $P$ when reordered according to $W_1$, and denote by $L_{\beta_j} \in \Pi_A$, $1\leq j\leq |A|$, the Lagrange polynomials with respect to  $\Pi_A$ and $P_A$
when reordered according to $W_2$. Denote by $S_{A,P} \in \R^{k\times |A|}$ the matrix given by the first $k$ rows of $R_{A,P}$. Then,
the Lagrange polynomials $\Lc_i$, $i =1 ,\dots,k$, are uniquely determined as
\begin{equation}\label{image}
   \Lc_i(x)= \sum_{j =0}^{|A|} c_{ij}L_{\beta_j}(x) \quad\text{with}\quad  C_i=(c_{ij})_{1\leq j\leq|A|} \in \R^{|A|} \quad \text{solving} \quad S_{A,P} C_i = e_i\,,
\end{equation}
where $e_i$ is the $i$-th standard basis vector of $\R^k$. Since $\rank(R_{A,P})=k$, the set $P_0$ is the maximal subset of $P$ with that property, yielding $(i)$.

For the second claim, consider $D_j = (d_{j,i})_{i=1,\dots,k}\in \R^{k}$, $1 \leq j\leq |A|-k$, with
$$ U_1 D_j = - U_2 e_j ,$$
where $e_j$ is the $j$-th standard basis vector of $\R^{|A|-k}$. Setting $(d_{j,i})_{i=1,\dots, |A|} := (D_j,-e_j) \in \R^{|A|}$ yields polynomials
\begin{equation} \label{kernel}
  \Mc_j(x) = \sum_{i=1}^{|A|} d_{j,i}L_{\beta_j}(x)
\end{equation}
that form a basis of $\mathfrak{m}_M=\{Q \in \Pi_A : Q(P) =0\} \cong \Pi_A / \Gamma_k(P_0)$. Therefore, the computational costs for
solving Eqs.~\eqref{RLU}, \eqref{image}, and \eqref{kernel} are all contained in $\Oc(|A|^3)$,
proving $(ii)$.

We use the fact that $\Gamma_k(P_0) \cong  \mathrm{span}\{\Lc_{i}\}_{i=1,\dots,k}$, which has already been proven in \linebreak
Eq.~\eqref{image}, to show $(iii)$. Indeed, we observe that $P_0\subseteq M$. Thus,
the restricted Lagrange polynomials remain linearly independent, and because $\mathrm{span}\{\Lc_{i | M}\}_{i=1,\dots,k} \subseteq \Pi_M$ we obtain
$\dim \Pi_M \geq \dim \Gamma_{k}(P_0) =k$. Consequently,
$$\Pi_A = \mathrm{span}\{\Lc_{i}\}_{i=1,\dots,k} + \mathrm{span}\{\Mc_{i}\}_{i=1,\dots,|A|-k} \cong \Gamma_k(P_0) + \mathfrak{m}_M$$
yields
$$\Gamma_{k}(P_0) \cong \Pi_A/\mathfrak{m}_M \supseteq   \{Q_{| M} : Q  \in \Pi_A/\mathfrak{m}_M \} \cong \{Q_{| M} : Q \in \Pi_A\} = \Pi_M$$
and therefore $\Gamma_{k}(P_0) \cong \Pi_M$, as claimed in $(iii)$.

We prove $(iv)$ by using a classic bases exchange argument \cite{maclane1936}:  We choose $p \in P_0' \setminus P_0$. Since $P_0' \subseteq Q_{M}^{-1}(0)$ and $Q_M(x)= \sum_{i=1}^k \Lc_{i}(x)- 1 =0$, $\forall x \in M$, there exists a
Lagrange polynomial $\Lc_{i_0}$, $1 \leq i_0\leq k$, with $\Lc_{i_0}(p) \neq 0$. We then set
$$\Lc_{i_0}^1(x) = \frac{1}{\Lc_{i_0}(p)}\Lc_{i_0}(x)  \in \Gamma_k(P_0)  \quad  \text{and} \quad  \Lc_{i}^1(x) = \Lc_{i}(x) - \Lc_{i}(p)\Lc_{i_0}^1(x) \in \Gamma_k(P_0)\,, \quad  i \neq i_0\,.$$
Then exchange $p_{i_0}$ with $p$, i.e., set $P_0^1 = (P_0 \setminus \{p_{i_0}\} ) \cup \{p\} = \{p_i^1\}_{i=1,\ldots,k}$ and observe that
$\Lc_i^1(p_j^1) = \delta_{i,j}$, $\forall 1\leq (i,j)\leq k$. Thus, we have constructed a Lagrange basis $\{\Lc_{i}^1\}_{i=1,\ldots,k} \subseteq \Gamma_{k}(P_0)$ w.r.t. $P_0^1$, implying that
$\Gamma_{k}(P_0) = \Gamma_{k}(P_0^1)$, and therefore $P_0^1$ is unisolvent w.r.t.~$\Gamma_{k}(P_0)$.
Setting $Q_{M'}(x)= \sum_{i=1}^k \Lc_{i}^1(x)$ yields $Q_{M'}(p_i^1)=1$ for all $i =1,\ldots,k$, implying $Q_{M'}(x) = Q_M(x) \equiv 0$ $\forall x \in M$.
Thus, $Q_M(x) - Q_{M'}(x) \in \mathfrak{m}_M$ holds due to $(iii)$.
By recursively continuing this exchange procedure (at most $k$ times),
we construct a Lagrange basis $\{\Lc_i^k\}_{i=1,\ldots,k}$ with respect to $P_0'=P_0^k$ within $\Gamma_k(P_0)$ that satisfies Eq.~\eqref{eq:QM}, proving $(iv)$.

$(v)$ follows from observing that because $A_1 \subseteq A_2$ we have $\Pi_{A_1} \subseteq \Pi_{A_2}$. Hence, $Q_{M_1}\in\Pi_{A_2}$ and
$Q_{M_1} (p) = Q_{M_2}(p) = 0$, $\forall p \in P$ imply that $Q_{M_1} - Q_{M_2} \in \mathfrak{m}_{M_2}$ due to $(iii)$.
Thus, $M_2 \subseteq Q_{M_1}^{-1}(0) = M_1$. Vice versa, projecting $Q_{M_2}$ onto $ \Pi_{M_1}$ yields $Q_{M_2} = \sum_{i=1}^k \Lc_i-1 =Q_{M_1}$ with $\Lc_i$ the Lagrange basis
spanning $\Pi_{M_1}$. Hence, $Q_{M_2} - Q_{M_1} \in \Pi_{A_2}/ \mathfrak{m}_{M_1}$ and thereby $M_1 \subseteq M_2$, proving the statement.
Finally, $(vi)$ follows directly from $(iv)$ and $(v)$.
\end{proof}

\begin{remark}[Uniqueness of the GPLS] We want to emphasise the importance of Theorem~\ref{theorem:Dual}$(iv,v,vi)$ stating that
regardless of the choice of polynomial degree, $A_1=A_{m,n_1,p_1}$, $A_2=A_{m,n_2,p_2}$, $n_1\leq n_2$, and of the points $P_1,P_2 \subseteq S$, the algebraic variety $M=M_1=M_2$ is uniquely determined
whenever $A_1 \subseteq A_2$ and $P_1 \cup P_2 \subseteq M_1\cap M_2$. Therefore, the approximation of any closed smooth surface $S\subseteq \R^3$ by an algebraic variety $S\approx M =Q_{M}^{-1}(0)$ is uniquely determined by the point set $P\subseteq S\cap M$ in that sense.
\label{rem:unique}
\end{remark}

\begin{definition}[Regular samples] Given an algebraic variety $M= Q_{M}^{-1}(0) \subseteq \R^3$ with $Q_{M} \in \Pi_{A_{m,n,p}}$ of $l_p$-degree at most $n\in \N$, we call a point set $P \subseteq M$ \emph{regular} if and only if there exists a subset $P_0 \subseteq P$ with $|P_0| =k \in \N$ and
$\Gamma_k(P_0) = \Pi_M$ from Eq.~\eqref{quotient}.
\label{def:regP}
\end{definition}
Since the associated matrix $R_{A,P}$, Eq.~\eqref{eq:RA}, has full rank with probability 1 for any uniformly random points $P\subseteq S$ \cite{sard1942measure,smale},
one can expect $P$ to be regular in practice whenever $P$ is of sufficient size.

\subsection{Global polynomial level sets for affine algebraic varieties}
Using the statements of Theorem~\ref{theorem:Dual}, we provide a numerical method for determining the GPLS approximation of a given affine algebraic variety, hence detailing contribution (C1) announced in the introduction.
A GPLS for an affine algebraic variety of sufficiently low degree can be given by:

\begin{corollary} \label{cor:Level} Let the assumptions of Theorem~\ref{theorem:Dual} be fulfilled, the bases \linebreak
  $\{\Lc_1,\dots,\Lc_{k}\} \subseteq \Gamma_k(P_0)$, $\{\Mc_1,\dots,\Mc_{|A|-k})\} \subseteq \Pi_{A}/\Gamma_k(P_0)$ from Theorem~\ref{theorem:Dual}$(ii)$ be computed,
  $R_{A,P} \in \R^{|P|\times |A|}$ as in Eq.~\eqref{eq:RA}, $M$ as in Theorem~\ref{theorem:Dual}$(iii)$, and $Q_M$ as in Eq.~\eqref{PiM}.
  \begin{enumerate}
    \item[i)] If $k =1$ then $M = \Mc_1^{-1}(0)$  and $Q_M = \lambda \Mc_1$ for some $\lambda \in \R \setminus\{0\}$.
    \item[ii)] Let $f : \Omega \lo \R$ be a (continuous) function.
    Then, the  Lagrange interpolant
    \begin{equation}\label{eq:Qfp}
      Q_{f,P_0,A} = \sum_{i=1,\dots,k} f(p_i)\Lc_i \in \Pi_M
    \end{equation}
     is uniquely determined in $\Pi_M \cong \Gamma_k(P_0)$ from Eq.~(\ref{quotient}).
  \end{enumerate}
\end{corollary}
\begin{proof} The proof follows directly from Theorem~\ref{theorem:Dual} and from the existence of optimal solutions to least squares problems, see e.g.,\cite{Lloyd_Num}.
\end{proof}

\begin{remark}
In the special case of Corollary~\ref{cor:Level}$(i)$, the GPLS of $M$ is
straightforwardly computed by deriving $\Mc_1$ according to Eq.~\eqref{kernel}. In Section~\ref{sec:Num}, we numerically demonstrate that this
approach provides an effective scheme for this class of surfaces.
\end{remark}

\begin{remark}\label{rem:fit}
  Consider the solution to
       the least squares problem
       \begin{equation}\label{eq:LS}
       C=(c_{\alpha})_{\alpha \in A} =\mathrm{argmin}_{X \in \R^{|A|}}\li\{\|R_{A,P} X -  F\|_2^2\re\}
       \end{equation}
       with $F = (f(p_i))_{i=1,\dots,k}\in \R^{|P|}$.
       Then $f \approx Q_{P,f,A} = \sum_{\alpha \in A} c_{\alpha}L_{\alpha} \in \Pi_A$. Moreover, up to the regression error, we have
       $Q_{P,f,A} - Q_{f,P_0,A} \in \mathfrak{m}_M$ with $Q_{f,P_0,A}$ from Eq.~\eqref{eq:Qfp} and $\mathfrak{m}_M$ as in
       Theorem~\ref{theorem:Dual}$(iii)$. Thus, in practice, it might be more convenient to derive the regressor $Q_{P,f,A}$ instead of the interpolant $Q_{f,P_0,A}$.
\end{remark}

\subsection{Global polynomial level sets for non-algebraic surfaces}\label{sec:SDF}

For non-algebraic surfaces $S \subseteq \R^3$, the matrix $R_{A,P} \in \R^{|P|\times |A|}$ in Eq.~\eqref{eq:RA} does not (sharply)
numerically separate into kernel (null space) and co-kernel. This makes direct computation of the bases $\{\Lc_1,\dots,\Lc_{k}\} \subseteq \Gamma_k(P_0)$, $\{\Mc_1,\dots,\Mc_{|A|-k})\} \subseteq \Pi_{A}/\Gamma_k(P_0)$ practically impossible. To resolve the issue, we introduce:

\begin{definition}[Relaxed signed distance  function]

Let $S \subseteq \R^3$ be a smooth closed surface, $P \subseteq S$ be a set of points on the surface, and $P_D = P_D^+\cap P_D^- \subseteq \Omega \setminus S$  arbitrary points  in some vicinity of $S$ to either side of the surface.
Given a smooth and strictly positive function $\mu : \Omega \lo \R^+$, we call
\begin{equation}\label{eq:Rdist}
   d(x) = \li\{\begin{array}{rl}
\mu(x)\dist(x,M)  & \text{if }  x \in P_D^+\\
-\mu(x)\dist(x,M)  & \text{if }  x \in P_D^-\\
 0 & \text{if }  x \in P \subseteq  M\\
\end{array}\re.
\end{equation}
a \emph{relaxed signed distance  function} with respect to $P_D$, where the relaxation factor $\mu(x)$ reflects the deviation from the proper signed-distance function.
\end{definition}

\begin{remark} Given a flat surface triangulation, see e.g., \cite{ruppert1995delaunay,shewchuk2002delaunay},
of the surface $S \subseteq \R^3$, a point set $P_D$ as described above can be generated by moving the vertices $V =P$ of the triangles
  along the mesh-normal field $\eta$, i.e., $q\in P \mapsto  q + D_q\eta = q'$, $D_q \in \R\setminus\{0\}$.
  Setting $d(q') = D_q$ yields a relaxed signed distance function as defined above.
\end{remark}

Given a relaxed signed distance function with respect to $P_D$,
we consider the node set $\bar P = P\cup P_D$ and extend $R_{A,P}$ in Eq.~\eqref{eq:RA} to $\bar R_{A,P} = (r_{i,\alpha})_{i=1,\ldots,|\bar P|, \alpha \in A} \in \R^{|\bar P| \times |A|}$ with $r_{i,\alpha} = L_{\alpha}(p_i)$, $p_i \in \bar P$.
The coefficients $C_A =(c_{\alpha})_{\alpha \in A}$ of a polynomial $Q_d \in \Pi_A$
\begin{equation}\label{eq:dist}
Q_{d}(x) = \sum_{\alpha \in A} c_{\alpha}L_{\alpha}(x) \approx d(x)
\end{equation}
approximating the relaxed signed distance  function can be derived by solving the least squares problem
$$C_A =\mathrm{argmin}_{X \in \R^{|A|}}\li\{\|\bar R_{A,P} X -  D\|_2^2\re\}\,, \quad D = (d(p_i))_{i=1,\ldots,|\bar P|}\,.$$

Level-set methods are most conveniently formulated in terms of the signed-distance function~\cite{sussman1998improved,sussman1999efficient}.
Here, we use a relaxed version to derive a GPLS $S\approx Q_{d}^{-1}(0)$ with non-vanishing gradient, i.e.,
  \begin{equation}
    \nabla Q_d (x)  \neq 0  \quad \text{for all} \quad x \in Q_{d}^{-1}(0) \,.
  \end{equation}
If the approximation $S \approx S'=Q_{d}^{-1}(0)$ is sufficiently close, the polynomial normal field $\eta =  \nabla Q_d/ \|\nabla Q_d \|$ enables
computing geometric entities of $S'$ with high (machine) precision, as demonstrated in Section~\ref{sec:Num_surf}.
The approximation quality, however, depends on the approximation power of the regression scheme, as addressed in the following section.

\section{Approximation theory}
\label{sec:App}

The above computational schemes derive GPLS approximations to algebraic varieties $M$ and non-algebraic surfaces $S$ from a regular surface point set $P\subseteq M$ using the statements of Theorem~\ref{theorem:Dual}.
If non-polynomial surfaces $S \subseteq \R^3$ are to be approximated, however, the question arises of how
accurate the GPLS approximation is. We address this question by using:

\begin{definition}[Lebesgue constant] \label{def:LEB} Let $m \in \N$, $A\subseteq \N^m$ be a downward closed set of multi-indices, $P_A\subseteq \Omega$ be a set of unisolvent nodes satisfying $(A1)$ from Definition~\ref{def:EA}.
Let  $f \in C^0(\Omega,\R)$ and $Q_{f,A}(x) = \sum_{\alpha \in A}f(p_\alpha)L_{\alpha}(x)$ be its Lagrange interpolant.
Then, we define the \emph{Lebesgue constant} analogously to the 1D case, see e.g. \cite{gautschi}, as
\begin{align*}
\Lambda(P_A) := \sup_{f\in C^0(\Omega,R)\,, \|f\|_{C^0(\Omega)}\leq 1} \|Q_{f,A}\|_{C^0(\Omega)}
             =  \Big\|\sum_{\alpha \in A} |L_{\alpha}|\Big\|_{C^0(\Omega)}\,.
\end{align*}
\end{definition}
Based on the 1D estimate
\begin{equation}\label{LEB}
 \Lambda(\Cheb_n)=\frac{2}{\pi}\big(\log(n+1) + \gamma +  \log(8/\pi)\big) + \Oc(1/n^2)\,,
\end{equation}
known for Chebyshev-Lobatto nodes, surveyed by \cite{brutman2}, \cite{cohen2,cohen3,MIP} further detail and study this concept in $m$D and show
that unisolvent nodes satisfying $(A2)$ from Definition~\ref{def:EA} induce high approximation power reflected in the small corresponding  Lebesgue constants.
Motivated by the classic Lebesgue inequality \cite{brutman2}, we deduce the following bound on the approximation error of the present regression scheme:

\begin{theorem}\label{theo:APP} Let the assumptions of Theorem~\ref{theorem:Dual}$(i-iv)$ be fulfilled and $M\subseteq \Omega$ be as in Theorem~\ref{theorem:Dual}$(iii)$; let further
  $P_{A_{m,n,p}}$, $n,m\in\N$,
  $p> 0$ be unisolvent nodes satisfying $(A1)$ from Definition~\ref{def:EA}, $f : \Omega \lo \R$ be a continuous function, and $f_{| M} : M \lo \R$
   its restriction to $M$. We
  denote by $Q_{f,A_{m,n,p}} = \sum_{\alpha \in A_{m,n,p}}f(p_\alpha)L_\alpha \in \Pi_{A_{m,n,p}}$ the Lagrange interpolant of $f$ in $P_{A_{m,n,p}}$ and by
  $$Q_{f,P_0,A_{m,n,p}} = \sum_{i=1,\dots,k} f(p_i)\Lc_i \in \Pi_M$$
  the polynomial interpolant of $f$ in $P_0$ according to Corollary~\ref{cor:Level}$(ii)$. Then, the approximation error is bounded by
\begin{equation}\label{eq:EST}
   \|f_{| M} - Q_{f,P_0,A_{m,n,p}}\|_{C^0(M)}  \leq (1+ \Lambda(P_A) \|S_{A,P}\|_\infty)\|f - Q_{f,A}\|_{C^0(\Omega)} + \mu\Lambda(P_A) \|S_{A,P}\|_\infty \,,
   % \leq (1+ \Lambda(P_{A_{m,n,p}}) \|S_{A_{m,n,p},P_0}\|_\infty)\|f - Q_{f,A_{m,n,p}}\|_{C^0(\Omega)}\,,
\end{equation}
where $S_{A_{m,n,p},P_0} \in \R^{|A_{m,n,p}|\times|P_0|}$ with $S_{A_{m,n,p},P_0}R_{A_{m,n,p},P_0}  = \mathrm{Id}_{\R^{|A_{m,n,p}| \times |A_{m,n,p}|}}$
is the \emph{Moore--Penrose pseudo-left-inverse}, see e.g., \cite{ben2003,Lloyd_Num} of the regression matrix $R_{A_{m,n,p},P_0}$ from Eq.~\eqref{eq:RA}
and
$$ \mu = \|\widetilde F - F \|_\infty \,, \quad \quad \widetilde F = (Q_{f,P_0,A}(p_i))_{i=1,\ldots,|P|}\,, \,F = (f(p_i))_{i=1,\ldots,|P|} \in \R^K$$
denotes the regression error.
\end{theorem}
\begin{proof} We shorten $A=A_{m,n,p}$. Due to Theorem~\ref{theorem:Dual}$(iv)$, $R_{A,P_0}$ has full $\rank\, R_{A,P_0} = |P_0|$. While the nodes $P_A$ are unisolvent with respect to $\Pi_A$ the interpolation operator
  $$ I_{P_A} : C^0(\Omega,\R) \lo \Pi_{A} \,, \quad f \mapsto Q_{f,A}$$
  is a linear operator with operator norm
  $$\|I_{P_A}\| = \sup_{f\in C^0(\Omega,R)\,, \|f\|_{C^0(\Omega)}\leq 1} \|Q_{f,A}\|_{C^0(\Omega)} = \Lambda(P_A)$$
  given by the Lebesgue constant
  from Definition~\ref{def:LEB}. In particular, $I_{P_A}(Q) = Q$ holds for all polynomials $Q \in \Pi_A$.
  Denote with
  $\mathbb{Q}_I = (Q_{f,A}(p))_{p \in P}$ the values of the interpolant in the data points $P$ and observe that the values of $Q_{f,P_0,A}$ in the interpolation nodes $P_A$ are given by
   $(Q_{f,P_0,A}(p_\alpha))_{\alpha \in A}\in \R^{|A|} =S_{A,P}\widetilde F$. Then we deduce:
    \begin{align*}
       \|f- Q_{f,P_0,A}\|_{C^0(M)} & \leq \|f - Q_{f,A}\|_{C^0(\Omega)} + \|Q_{f,A}-Q_{f,P_0,A}\|_{C^0(\Omega)} \\
       &\leq \|f - Q_{f,A}\|_{C^0(\Omega)} + \|I_{P_A}(Q_{f,A}-Q_{f,P_0,A})\|_{C^0(\Omega)} \\
        &\leq \|f - Q_{f,A}\|_{C^0(\Omega)} + \Lambda(P_A)\|S_{A,P}(\mathbb{Q}_I- \widetilde F)\|_{\infty} \\
       % & \leq \|f - Q_{f,A}\|_{C^0(\Omega)} + \Lambda(P_A) \|S_{A,P}\|_\infty\|\mathbb{Q}_I - \widetilde F\|_{\infty}  \\
        & \leq \|f - Q_{f,A}\|_{C^0(\Omega)} + \Lambda(P_A) \|S_{A,P}\|_\infty\big (\|F-\mathbb{Q}_I\|_{\infty} + \|F-\widetilde F\|_{\infty}   \big) \  \\
       & \leq (1+ \Lambda(P_A) \|S_{A,P}\|_\infty)\|f - Q_{f,A}\|_{C^0(\Omega)} + \mu\Lambda(P_A) \|S_{A,P}\|_\infty  \,,
    \end{align*}
    where we used $\|F -\mathbb{Q}_I\|_{\infty} \leq \|f - Q_{f,A}\|_{C^0(\Omega)}$
    for the last estimate.
\end{proof}

The statement implies the following consequence:
\begin{corollary}\label{cor:uni1} Let $m,n,p \in \N$, $A=A_{m,n,p}\subseteq \N^m$ be a downward closed set of multi-indices, $P_A\subseteq \Omega$ be a set of unisolvent nodes satisfying $(A1)$ from Definition~\ref{def:EA} in dimension $m=3$.
  Given are a closed smooth surface $S \subseteq \R^3$, a regular point set $P_n =\{p_0,\dots,p_{K_n}\}\subseteq S$, $K_n \geq |A_{m,n,p}|$, and
  a continuous function $f : S \lo \R$ possessing an (analytic) extension to a function $\widetilde f : \Omega \lo \R$ such that
  \begin{equation}\label{eq:CON}
    \|\widetilde f - Q_{\widetilde f,A_{m,n,p}}\|_{C^0(\Omega)}  = o(1 +\Lambda(P_{m,n,p}) \|S_{A_{m,n,p},P_n}\|_\infty) \,.
  \end{equation}
  Given that the regression error $\mu$ from Theorem~\ref{theo:APP} tends to zero fast, $\mu \in o(\Lambda(P_A) \|S_{A,P}\|_\infty)$, the sequence of polynomial interpolants $Q_{f,P_{0,n},A_{m,n,p}}$ from Theorem~\ref{theorem:Dual}$(i)$ approximate $f$, i.e.,
  \begin{equation*}
    Q_{f,P_{0,n},A_{m,n,p}} \xrightarrow[n\rightarrow \infty]{} f  \quad \text{uniformly on}\,\,\, S\,.
  \end{equation*}
\end{corollary}
\begin{proof} The proof follows from Theorem~\ref{theo:APP}.
\end{proof}

While the choice of Leja-ordered Chebyshev-Lobatto nodes, $(A2)$ in Definition~\ref{def:EA}, results in small Lebesgue constants  $\Lambda(P_{m,n,p})$ \cite{MIP},
the question of which functions $f : S \lo \R$ can be expected to satisfy the condition in Eq.~\eqref{eq:CON} remains.
To answer this question, we first summarise recent results by \cite{Lloyd2,converse} that provide a deeper insight:

Let
$E_{m,h^2}^2$ be the \emph{Newton ellipse} with foci $0$ and $m$
and leftmost point $-h^2$. For $m \in \N$ and $h \in [0,1]$, we set $\rho = h + \sqrt{1 + h^2}$
and call the open region
\begin{equation}\label{Tdomain}
   N_{m,\rho} = \li\{ (z_1,\dots,z_m) \in \C^m :  (z_1^2 + \cdots  + z_m^2) \in E_{m,h^2}^2 \re\}
 \end{equation}
the \emph{Trefethen domain} \cite{Lloyd2}. We call a  continuous function $f : \Omega \lo \R$ a \emph{Trefethen function} if
$f  =\sum_{\alpha \in \N^m}c_\alpha \prod T_{\alpha_i}\in \Pi_{A_{m,n,p}}$ can be expanded in an absolute convergent Chebyshev series on $\Omega$
and in addition can be \emph{analyticaly  extended}
to the Trefethen domain $N_{m,\rho}\subseteq \C^m$ of radius  $\rho >1$.
In \cite{Lloyd2} Trefethen proved an upper bound on the convergence rate for truncating the Trefethen function
$\mathcal{T}_{A_{m,n,p}}(f)  =\sum_{\alpha \in A_{m,n,p}}c_\alpha \prod T_{\alpha_i}\in \Pi_{A_{m,n,p}}$
to the polynomial space $\Pi_{A_{m,n,p}}$:
\begin{equation}\label{Rate}
  \| f - \mathcal{T}_{A_{m,n,p}}(f)\|_{C^0(\Omega)}  = \li\{\begin{array}{ll}
                                                                      \Oc_\ee(\rho^{-n/\sqrt{m}})  & \quad p =1 \\
                                                                      \Oc_\ee(\rho^{-n}) & \quad p =2\\
                                                                      \Oc_\ee(\rho^{-n})  & \quad p =\infty\, ,
                                                                      \end{array}
\re.
\end{equation}
where  $g \in \Oc_\ee(\rho^{-n})$ if and only if $g \in \Oc((\rho-\ee)^{-n})$ $\forall \ee >0$.

This suggests that interpolation or regression with respect to \emph{Euclidean $l_2$-degree} or \emph{maximum $l_\infty$-degree} can achieve faster convergence rates than interpolation with respect to
\emph{total $l_1$-degree}, with $l_2$-degree requiring less coefficients than $l_\infty$-degree, see Eq.~\eqref{eq:size}.
If $f$ is a (relaxed) signed distance function, as in  section~\ref{sec:SDF}, we
therefore find the following consequence of Corollary~\ref{cor:uni1}:

\begin{remark}\label{rem:FIT}  Given a surface $S \subseteq \Omega$ and a regular point set $P \subseteq S$, assume there exists a smooth relaxed signed distance function
$d : \Omega \supseteq S \lo \R$, which in addition also is a Trefethen function for which the optimal (Euclidean) rate in Eq.~\eqref{Rate} applies with radius $\rho >1$.
Thus:
\begin{equation}\label{rem:APP}
  \|d - Q_{d,A_{3,n,2}}\|_{C^0(\Omega)}  \in  \Oc(\rho^{-n}) \quad \text{and} \quad \rho^{-n} \in \omicron(1 +\Lambda(P_{3,n,p})\|S_{A_{3,n,p},P}\|_\infty)\,.
\end{equation}
Then, the surface $S \approx Q_{d,A_{3,n,2}}^{-1}(0)$ can be uniformly approximated by fitting $d$ according to Eq.~\eqref{eq:dist}.
\end{remark}
Because Trefethen functions are a general class of analytic functions \cite{Lloyd2}, the numerical experiments in Section~\ref{sec:Num_surf} suggest
that Eq.~\eqref{rem:APP} holds for a larger set of smooth closed surfaces $S$.

\section{Curvatures and differential operators on polynomial hypersurfaces}

Once a \linebreak
GPLS approximation of a surface has been determined, differential geometric quantities can be computed analytically. We provide explicit formulas for computing mean curvature, Gauss curvature, and the Laplacian of mean curvature.

We consider the affine algebraic variety $M \subseteq \R^3$ as an iso-hypersurface of a GPLS $Q_{M}^{-1}(0) = M$, as in Theorem~\ref{theorem:Dual}.
In order to provide explicit formulas for basic geometric quantities of $M$, we follow \cite{goldman2005} in $\R^3$ with standard inner product $\li< e_i,e_j\re> = \delta_{i,j}$ and standard basis $\{e_i\}_{i=1,\ldots,3}$.

\subsection{Mean and Gauss curvature}

The \emph{gradient} $\nabla Q_M = (\p_x Q_M,\, \p_y Q_M,\, \p_z Q_M) \in \R^3$ and the \emph{Hessian} $H_M =\nabla(\nabla Q_M) \in \R^{3 \times 3}$ of $Q_M$
\begin{equation*}
 H_M  = \li(\begin{array}{ccc}
 \frac{\p^2 Q_M}{\p_{x}^2} & \frac{\p^2 Q_M}{\p_{x}\p_y} & \frac{\p^2 Q_M}{\p_{x}\p_z} \\
 \frac{\p^2 Q_M}{\p_{y}\p_x} & \frac{\p^2 Q_M}{\p_y^2} & \frac{\p^2 Q_M}{\p_{y}\p_z} \\
 \frac{\p^2 Q_M}{\p_z\p_{x}} & \frac{\p^2 Q_M}{\p_{z}\p_y} & \frac{\p^2 Q_M}{\p_z^2} \\
 \end{array}\re)
\end{equation*}
are the main ingredients for the following computations. Both Gauss and mean curvature can be computed from these quantities \cite{goldman2005} as:
\begin{align}
  K_{\mathrm{Gauss}} &= \frac{\det \li(\begin{array}{cc}H_M  & \nabla Q_M^T \\  \nabla Q_M^T & 0 \end{array}\re)}{\|\nabla Q_M\|^4} \label{eq:GC}\\
  K_{\mathrm{mean}} &= \frac{\nabla Q_M H_M \nabla Q_M^T - \|\nabla Q_M\|^2\mathrm{trace}(H_M)}{2\|\nabla Q_M\|^3}\,. \label{eq:MC}
\end{align}
While there are several alternative formulas \cite{goldman2005}, the above two allow for stable and numerically accurate evaluation, as we demonstrate in Section~\ref{sec:Num_curv}.

\subsection{The Laplacian of mean curvature}\label{sec:LB}

The algebraic variety $M=Q_{M}^{-1}(0)$ of the GPLS, together with its unit normal field $\eta= \nabla Q_M / \|\nabla Q_M\|$, enables computing covariant derivatives and, therefore, the surface-intrinsic gradient and the \emph{Laplace-Beltrami operator} of a function $f: M \lo \R$ as:
\begin{align*}
\nabla_M f &=  \nabla f - \li<\eta,\nabla f\re>\eta\\
\upDelta_M f & = \upDelta f + 2K_{\mathrm{mean}} \left<\eta, \nabla f \right> - \li<\eta,  \nabla ^2 f \cdot \eta\re>\,,
\end{align*}
where $\nabla^2 f$ denotes the Jacobian of the gradient of $f$ \cite{reilly1982mean,xu2003eulerian}.

Computing the intrinsic Laplacian of mean curvature, a 4$^{\text{th}}$-order differential term of the surface, is required in many applications, including surface diffusion \cite{sethian1999motion,smereka2003,greer2006fourth,seifert1997configurations}, and turns out to mostly be the bottleneck in terms
of accuracy and runtime performance.
For a GPLS $M=Q_{M}^{-1}(0)$ with unit normal field $\eta= \nabla Q_M / \|\nabla Q_M\|$, an analytical identity can be derived by splitting the formula for mean curvature into two parts
$$ K_{\mathrm{mean}} = \frac{1}{2}\left( \nabla Q_M H_M \nabla Q_M^T - \|\nabla Q_M\|^2\mathrm{trace}(H_M) \right) \cdot  \frac{1}{\|\nabla Q_M\|^3} =:u\cdot v $$
and computing:
\begin{align}
\upDelta_M K_{\mathrm{mean}} &= \upDelta (uv) + 2K_{\mathrm{mean}} \left<\eta, \nabla (uv) \right> - \li<\eta,  \nabla ^2 (uv) \eta\re>\, \nonumber \\
&= u \upDelta v + 2 \left<\nabla u, \nabla v\right> + v \upDelta u  + 2 (u v) \left<\eta, u \nabla v + v \nabla u \right> \nonumber \\
&- \left< \eta, (u \nabla^2 v + \nabla u \otimes \nabla v + \nabla v \otimes \nabla u + v \nabla^2 u) \eta \right>\,. \label{eq:lapMC}
\end{align}
Numerical experiments involving these computations are shown in section~\ref{sec:Num_Lap}.

\section{Numerical Experiments}\label{sec:Num}

%---- HERE
We implemented the present GPLS approach
based on Theorem~\ref{theorem:Dual} in Python as part of the package {\sc minterpy} \cite{minterpy}. In the following numerical experiments, we benchmark our implementation in comparison with two related alternative methods:
\begin{enumerate}
    \item[B1)] Curved Finite Elements (CFE): This method uses a  \emph{polygonal surface mesh} and curved finite elements to locally approximate the surface for each mesh element with a polynomial of degree 7. The method is implemented using {\sc DUNE~2.7.0} \cite{sander2020}.
    \item[B2)] Closest-Point Finite Differences (CP-FD): This method combines
       the \emph{closest-point extension} of a \emph{level set} with local finite-difference stencils for polynomial interpolation~\cite{saye2014high}. The method is implemented using {\sc OpenFPM}~\cite{incardona2019openfpm}.
\end{enumerate}
All numerical experiments were run on a standard Linux laptop (Intel(R) Core(TM) i7-1065G7 CPU @1.30GHz, 32\,GB RAM) within reasonable time (seconds up to several minutes). Unless specified otherwise, we use the nodes $P_{A_{3,n,2}} \subseteq \Omega$, $n \in \N$, $p=2$, which fulfil $(A2)$ from Definition~\ref{def:EA}.

\subsection{Approximation of algebraic varieties}

We start by comparing the three methods on the basic task of approximating several classic affine algebraic varieties $M \subseteq \R^3$ as
given by the following (global) parametrisations:
\begin{enumerate}
  \item[S1)] Ellipsoid \quad $\frac{x^2}{a^2} + \frac{y^2}{b^2} + \frac{z^2}{c^2} = 1$,\quad  $a,b,c \in \R\setminus\{0\}$.
  \item[S2)] Biconcave disc \quad $(d^2 + x^2 + y^2 + z^2)^3 - 8d^2(y^2 + z^2) - c^4 = 0$,  \quad $c<d \in \R\setminus\{0\}$.
  \item[S3)] Torus \quad $(x^2 + y^2 + z^2 + R^2 - r^2)^2 - 4R^2(x^2 + y^2) =0$, \quad $0 <r <R \in \R$
  \item[S4)] Genus 2 surface \quad $2y(y^2 - 3x^2)(1 - z^2) + (x^2 + y^2)^2 - (9z^2 - 1)(1 - z^2) = 0$
  \item[S5)] Klein bottle \quad  $(x^2 + y^2 + z^2 + 2y - 1)\big((x^2 +  y^2 + z^2 - 2y - 1)^2 - 8z^2\big)+ 16xz(x^2 + y^2 + z^2  - 2y - 1) = 0$.
\end{enumerate}

\begin{figure}[t!]
\includegraphics[width=\textwidth]{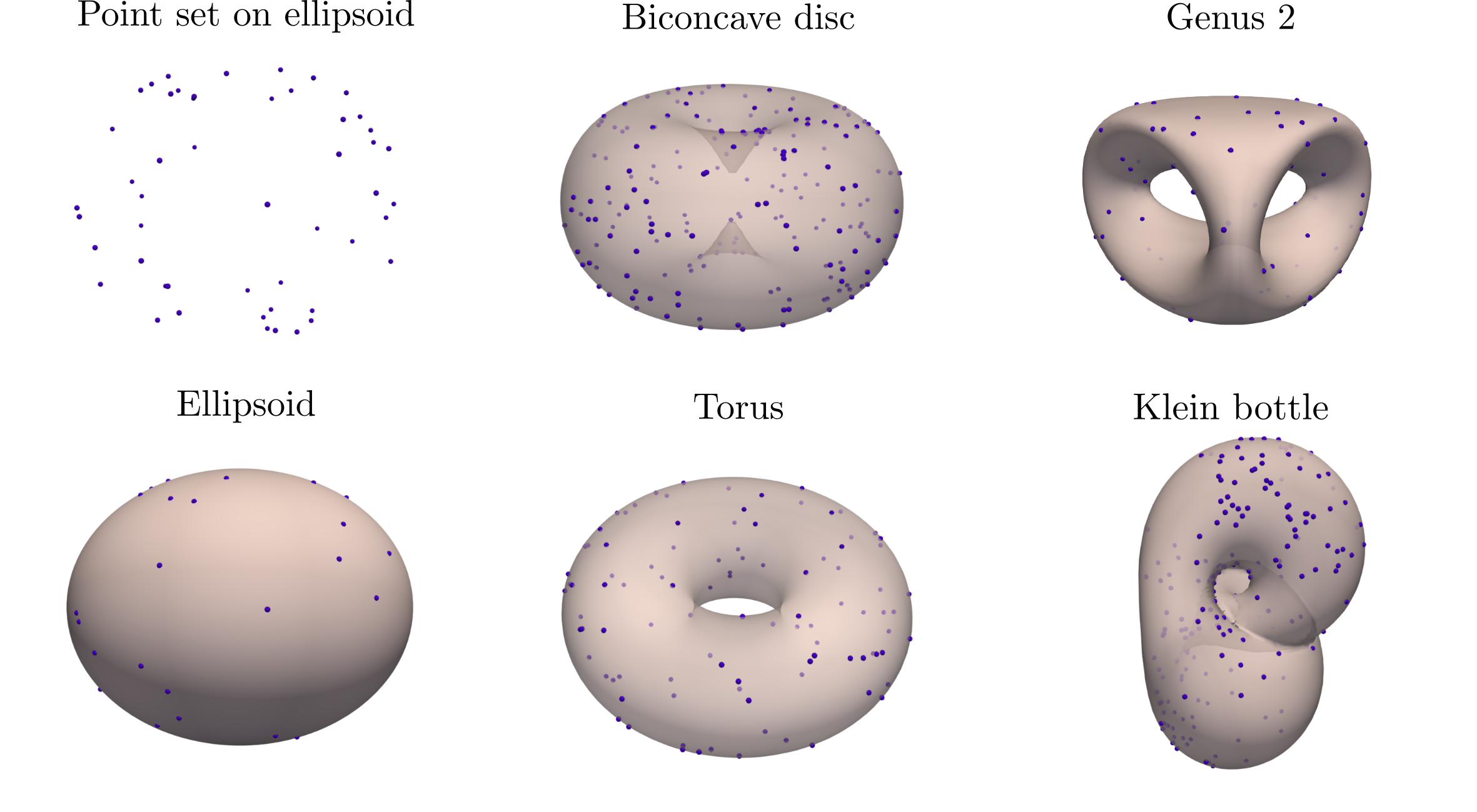}
%IFS: fix typo in figure: "pointset" --> "point set"
%MH: Done by Sachin!
\caption{The GPLS $Q^{-1}_M(0)$ derived from randomly sampled points (blue dots) on the five algebraic test surfaces (S1)--(S5).
\label{fig:surface_fitting}}
\end{figure}

\begin{experiment}[Surface reconstruction from regular point sets]\label{exp:REC} We sample $N\in \N$ random points
$P \subseteq S$, $|P|=N$, on each surface $S=M$ given by the algebraic varieties above,
as visualised in Fig.~\ref{fig:surface_fitting}.
All point positions are stored with machine precision (32-bit double-precision arithmetics), i.e., the formulas above hold for all $q \in P$ with an accuracy of $\approx 10^{-15}$.

When considering the multi-indices $A_{3,n,2}$ with $n =\deg(M) \in \N$ equal to the degree of the corresponding variety, then Corollary~\ref{cor:Level}$(i)$ applies to all algebraic varieties,
allowing us to compute the polynomial $Q_M \in \Pi_A$ with $M = Q_M^{-1}(0)$ using the GPLS method with the surface points $P$.
The quality of the GPLS approximation is measured for each true surface point $q \in P$ by
computing the shortest distance $\mathrm{dist}(q,M) \in \R^+$ to the GPLS surface when following the GPLS normal $\eta(x) = \nabla Q_M(x) / \|\nabla Q_M(x)\|$ due to classic Newton-gradient-descent till reaching $Q_M(q+ d_q\eta(q)) = 0$ (with machine precision).
\label{exp:SF}
\end{experiment}
%IF:
%MH: Above is the error measure described ...caused confusion prior.

The $L_{\infty}$-norm across all surface points and the number of points used ($N$) are reported in Table~\ref{tbl:fitting_errors} (columns ``surface fitting'' and ``$N$'').
We observe that the GPLS method approximates all surfaces, including the non-orientable,
self-intersecting Klein bottle, with an accuracy close to machine precision.
Several repetitions of the experiment for
different samples of random surface points produced comparable results differing in accuracy by less than one order of magnitude.
The same is true when measuring the fitting error on 100 randomly sampled test points $P_{\mathrm{test}}\subseteq S\setminus P$ that were not used for computing the GPLS.

\begin{table}[t!]
\centering
\begin{tabular}{lllr}
2D surface  & \multicolumn{2}{c}{$L_\infty$ error} & $N$  \\
 & surface fitting & coefficients & \\
\hline
Ellipsoid ($a = 0.8, b = 0.9, c = 1.0$)     & $8.96 \cdot 10^{-16}$ & $2.52 \cdot 10^{-15}$ & $50$\\
Biconcave disc ($d = 0.5, c = 0.375$)       & $9.90 \cdot 10^{-15}$ & $3.31 \cdot 10^{-6}$ & $200$ \\
Torus ($R = 0.5, r = 0.3$)                  & $1.13 \cdot 10^{-14}$ & $1.05 \cdot 10^{-12}$ & $100$ \\
Genus 2 surface                             & $1.19 \cdot 10^{-14}$ & $2.60 \cdot 10^{-11}$ & $100$\\
Klein bottle                                & $1.95 \cdot 10^{-12}$ & $1.58 \cdot 10^{-9}$  & $200$
\end{tabular}
\caption{Reconstruction errors for the GPLS method with different numbers of uniformly random surface points $N$ on the 2D surfaces given by the algebraic varieties shown in Fig.~\ref{fig:surface_fitting}.}
\label{tbl:fitting_errors}
\end{table}

\begin{experiment}[Coefficients reconstruction] We consider the GPLS $Q_M \in \Pi_A$ from Experiment~\ref{exp:REC} in canonical form $Q_M = \sum_{\alpha \in A_{3,n,2}}d_\alpha x^\alpha$ and normalise
$\widetilde Q_M = \lambda Q_{M}$ with $\lambda \in \R$ so that the leading coefficient $d_{\alpha}\in \R$, $\alpha = \mathrm{argmax}_{\alpha \in A_{3,n,2}} \{d_\alpha \neq 0\}$  coincides with the leading coefficient $c_\alpha$ of the original surface parametrisation polynomial $Q_S$, $S =(S1), \ldots, (S5)$ in canonical form. According to Corollary~\ref{cor:Level}$(i)$, the two polynomials
have to be identical, i.e., $\widetilde Q_M =Q_S$.
The $L_\infty$ difference $\|D-C\|_{\infty}$, $D =(d_{\alpha})_{\alpha \in A_{3,n,2}}$, $C =(c_{\alpha})_{\alpha \in A_{3,n,2}}$, of the GPLS and ground-truth coefficients is reported in Table~\ref{tbl:fitting_errors} (column ``coefficients'').
\label{exp:form}
\end{experiment}
Apart from the biconcave disc, all polynomial formulas are recovered close to machine precision.
The lower accuracy reached for the biconcave disc reflects its relatively high polynomial degree $n =6$, which makes representations in canonical polynomial basis imprecise.
% However, the corresponding fitting errors indicate that the coefficients were determined sufficiently accurately for approximating the surface. The coefficients with maximum do not have a big impact on the value $Q_M(x_0)$, $x_0 \in \Omega$.

Together, the results of Experiments~\ref{exp:REC} and \ref{exp:form} validate the GPLS method for computing global level-set surface approximations from regular point samples on (low-degree) algebraic surfaces.

\begin{table}[t!]
\centering
\begin{tabular}{lccr}
& \multicolumn{2}{c}{$L_\infty$ curvature error} & $N$  \\
Global Polynomial Level Set (GPLS) & $K_{\mathrm{mean}}$ & $K_{\mathrm{Gauss}}$ &  \\
\hline
Ellipsoid ($a = 1.0, b = 1.0, c = 1.0$) & $1.78 \cdot 10^{-15}$  & $3.55 \cdot 10^{-15}$  & $50$\\
Ellipsoid ($a = 1.0, b = 1.0, c = 0.6$) & $1.78 \cdot 10^{-15}$  & $5.77\cdot 10^{-15}$  & $50$\\
Ellipsoid ($a = 0.6, b = 0.6, c = 1.0$) & $3.33 \cdot 10^{-15}$  & $1.24\cdot 10^{-14}$  & $50$\\
Ellipsoid ($a = 0.6, b = 0.8, c = 1.0$) & $3.11 \cdot 10^{-15}$  & $7.11\cdot 10^{-15}$  & $50$\\
\hline
Biconcave disc ($d = 0.5, c = 0.375$) & $1.46 \cdot 10^{-10}$ & $6.38 \cdot 10^{-10}$ & $200$ \\
Biconcave disc ($d = 0.5, c = 0.4$) & $5.21 \cdot 10^{-11}$ & $2.42 \cdot 10^{-10}$ & $200$ \\
Biconcave disc ($d = 0.4, c = 0.2$) & $9.24 \cdot 10^{-11}$ & $7.78 \cdot 10^{-11}$ & $200$ \\
\hline
Torus ($R = 0.5, r = 0.3$) & $3.69 \cdot 10^{-13}$ & $2.21 \cdot 10^{-12}$ & $100$ \\
Torus ($R = 0.4, r = 0.3$) & $4.89 \cdot 10^{-13}$ & $6.73 \cdot 10^{-12}$ & $100$ \\
Torus ($R = 0.5, r = 0.1$) & $4.70 \cdot 10^{-12}$ & $1.71 \cdot 10^{-11}$ & $100$ \\
\hline
Genus 2 surface & $9.40 \cdot 10^{-13}$ & $3.46 \cdot 10^{-12}$ & 100 \\
 & & & \\
Curved Finite Elements (CFE)  & & & \\
\hline
Ellipsoid ($a = 1.0, b = 1.0, c = 1.0$) & $2.23 \cdot 10^{-7}$   & -  & $54722$\\
Ellipsoid ($a = 1.0, b = 1.0, c = 0.6$) & $8.66 \cdot 10^{-7}$   & -  & $46082$\\
Ellipsoid ($a = 0.6, b = 0.6, c = 1.0$) & $1.07 \cdot 10^{-7}$  & -  & $46082$\\
Ellipsoid ($a = 0.6, b = 0.8, c = 1.0$) & $3.57 \cdot 10^{-7}$   & -   & $46082$\\
\hline
Torus ($R = 0.5, r = 0.3$) & $8.7 \cdot 10^{-7}$ & - & $463680$ \\
Torus ($R = 0.4, r = 0.3$) & $5.29 \cdot 10^{-7}$ & - & $124800$ \\
Torus ($R = 0.5, r = 0.1$) & $4.06 \cdot 10^{-6}$ & - & $79680$ \\
& & & \\
Closest-Point Finite Differences (CP-FD) & & & \\
\hline
Ellipsoid ($a = 1.0, b = 1.0, c = 1.0$) & $1.31 \cdot 10^{-7}$ & $2.63 \cdot 10^{-7}$ & $725912$ \\
Ellipsoid ($a = 1.0, b = 1.0, c = 0.6$) & $2.45 \cdot 10^{-5}$ & $6.16 \cdot 10^{-5}$ & $543632$ \\
Ellipsoid ($a = 0.6, b = 0.6, c = 1.0$) & $3.55 \cdot 10^{-5}$ & $1.93 \cdot 10^{-4}$ & $383600$ \\
Ellipsoid ($a = 0.6, b = 0.8, c = 1.0$) & $2.10 \cdot 10^{-5}$ & $6.85 \cdot 10^{-5}$ & $461664$ \\
\hline
Torus ($R = 0.5, r = 0.3$) & $4.82 \cdot 10^{-5}$ & $2.96 \cdot 10^{-4}$ & $770080$ \\
Torus ($R = 0.4, r = 0.3$) & $1.28 \cdot 10^{-3}$ & $9.55 \cdot 10^{-3}$ & $616176$ \\
Torus ($R = 0.5, r = 0.1$) & $1.74 \cdot 10^{-3}$ & $1.08 \cdot 10^{-2}$ & $257056$ \\
& & & \\
\end{tabular}
\caption{Errors of curvature computations using the three methods (GPLS, CFE, CP-FD) for orientable algebraic surfaces represented using different numbers of points $N$.}
\label{tbl:curv_errors}
\end{table}

\subsection{Mean and Gauss curvatures of algebraic varieties}\label{sec:Num_curv}

After having validated the surface approximation properties of the GPLS method, we test how accurately differential geometric quantities of the surface can be computed from the GPLS parametrisation. We first consider Gauss and mean curvature, which are
2$^{\text{nd}}$-order derivatives, before looking at the 4$^{\text{th}}$-order Laplacian of curvature in the subsequent section.

For ellipsoids and tori, the analytical expressions are known:
\begin{enumerate}
  \item[S1)] Ellipsoid \quad $K_{\mathrm{mean}} = \frac{|x^2 + y^2 + z^2 - a^2 - b^2 - c^2|}{2(abc)^2 (\frac{x^2}{a^4} + \frac{y^2}{b^4} + \frac{z^2}{c^4})^{3/2}}$ \quad and \quad  $K_{\mathrm{Gauss}}= \frac{1}{(abc)^2 \left( \frac{x^2}{a^4} + \frac{y^2}{b^4} + \frac{z^2}{c^4} \right)^2}$.
  \item[S3)] Torus \quad \quad $K_{\mathrm{mean}}  = \frac{R + 2r \cos \theta}{2r(R + r \cos \theta)}$ \quad and \quad $K_{\mathrm{Gauss}} = \frac{\cos \theta}{r(R + r \cos \theta)}$, where we used
   toric coordinates $(x,y,z) = (R + r \cos \theta) \cos \varphi,\, (R + r \cos \theta) \sin \varphi,\, r \sin \theta)$, $\varphi,\theta \in [0,2\pi)$.
\end{enumerate}
Analytic expressions for the biconcave disc and the genus 2 surface also exist. However, for the sake of simplicity, we used {\sc Mathematica 11.3} for
the ground-truth computations in these cases.

\begin{table}[t]
\centering
\begin{tabular}{lcr}
& $L_\infty$ error & $N$  \\
Global Polynomial Level Set (GPLS) & $\upDelta_{M} K_{\mathrm{mean}}$ &  \\
\hline
Ellipsoid ($a = 1.0, b = 1.0, c = 1.0$) & $2.09 \cdot 10^{-11}$  & $50$\\
Ellipsoid ($a = 1.0, b = 1.0, c = 0.6$) & $4.93 \cdot 10^{-11}$  & $50$\\
Ellipsoid ($a = 0.6, b = 0.6, c = 1.0$) & $8.08 \cdot 10^{-11}$  & $50$\\
& & \\
Closest-Point Finite Differences (CP-FD) & & \\
\hline
Ellipsoid ($a = 1.0, b = 1.0, c = 1.0$) & $1.18 \cdot 10^{-3}$ & $725912$ \\
Ellipsoid ($a = 1.0, b = 1.0, c = 0.6$) & $1.59 \cdot 10^{-1}$ & $543632$ \\
Ellipsoid ($a = 0.6, b = 0.6, c = 1.0$) & $3.00 \cdot 10^{-1}$ & $383600$ \\
& &
\end{tabular}
\caption{Maximum errors of the Laplacian of mean curvature $\upDelta _M K_{\mathrm{mean}}$ computed using GPLS and CP-FD for axisymmetric ellipsoids.}
\label{tbl:lap_curv_errors}
\end{table}

\begin{experiment}[Curvature computation] We consider only the orientable surfaces from Experiment~\ref{exp:REC} and compute their
mean and Gauss curvatures from the GPLS approximation according to Eqs.~\eqref{eq:GC} and \eqref{eq:MC}. We compare the results with those computed using the CFE and CP-FD methods.
While GPLS can compute the curvatures once $Q_M$ is determined, CFE and CP-FD rely on feasible computational meshes or grids.
Some of the benchmark computations for those methods therefore had to be skipped due to incommensurate implementation effort.
Curvature errors are measured at each surface/grid point and the $L_\infty$ norm reported  in Table~\ref{tbl:curv_errors} along with the total number of surface/grid points $N$ used by the methods.
\end{experiment}

The curvatures computed by GPLS are seven to eight orders or magnitude more accurate than those computed using either CFE or CP-FD methods. In some cases, the GPLS reaches machine precision. The computational cost in terms of the number of points $N$ required is also orders of magnitude better for GPLS than for CFE and CP-FD.

Moreover, GPLS is the only method that allows evaluating curvature formulae at any location $x_0 \in M$. This allows us to compute the GPLS errors
at the points used by CFE and CP-FD, respectively. The resulting GPLS accuracies deviate by  less than an order of magnitude from those reported in Table~\ref{tbl:curv_errors} on the points used to derive the GPLS.

\subsection{Laplacian of mean curvature}\label{sec:Num_Lap}

Next, we consider computing a 4$^{\text{th}}$-order differential quantity of the surfaces, the Laplacian of mean curvature. The reference values for axisymmetric ellipsoidal surfaces (with $a = b$)  were computed using {\sc Mathematica 11.3}.

\begin{experiment}[Laplacian of mean curvature] We compute the surface Laplacian \linebreak
  $\upDelta _M K_{\mathrm{mean}}$ of the mean curvature using Eq.~\eqref{eq:lapMC} for a GPLS.
While Gauss curvature depends non-linearly on the Hessian, Eq.~\eqref{eq:GC} computing this quantity using CFE is not straightforward, which is why a direct comparison is omitted.
%IFS: what exactly was the problem?
%MH: We have to ask Gentian
%IFS: Could you please ask him? It seems strange to me.
The results computed using the CP-FD and GPLS methods are reported in Table~\ref{tbl:lap_curv_errors}.
\end{experiment}

Also for the Laplacian of mean curvature, the GPLS results are orders of magnitude more accurate than the CP-FD ones, while using much fewer surface points. However, both methods lose about 4 orders of magnitude in precision compared to computing mean curvature alone (cf.~Table \ref{tbl:curv_errors}, where the same surfaces were considered).

\subsection{Non-algebraic surfaces}\label{sec:Num_surf}
In order to test the GPLS approach on non-algebraic surfaces, we consider the well-known example surface $S_B$
given by the
  \emph{Stanford Bunny} dataset\footnote{available from http://graphics.stanford.edu/data/3Dscanrep/} containing 35,947 surface points with associated surface-normal vectors. Fig.~\ref{fig:bunny_fitting}. We complement our investigations by considering the \emph{Spot dataset}\footnote{available from https://www.cs.cmu.edu/~kmcrane/Projects/ModelRepository/}, Fig.~\ref{fig:cow_fitting}.
  %IFS: Let's add the Utah teapot here!
  %MH: We add the cow !

  \begin{figure}[t!]
    \centering
\begin{tabular}{ll}
    \includegraphics[width=0.4\textwidth]{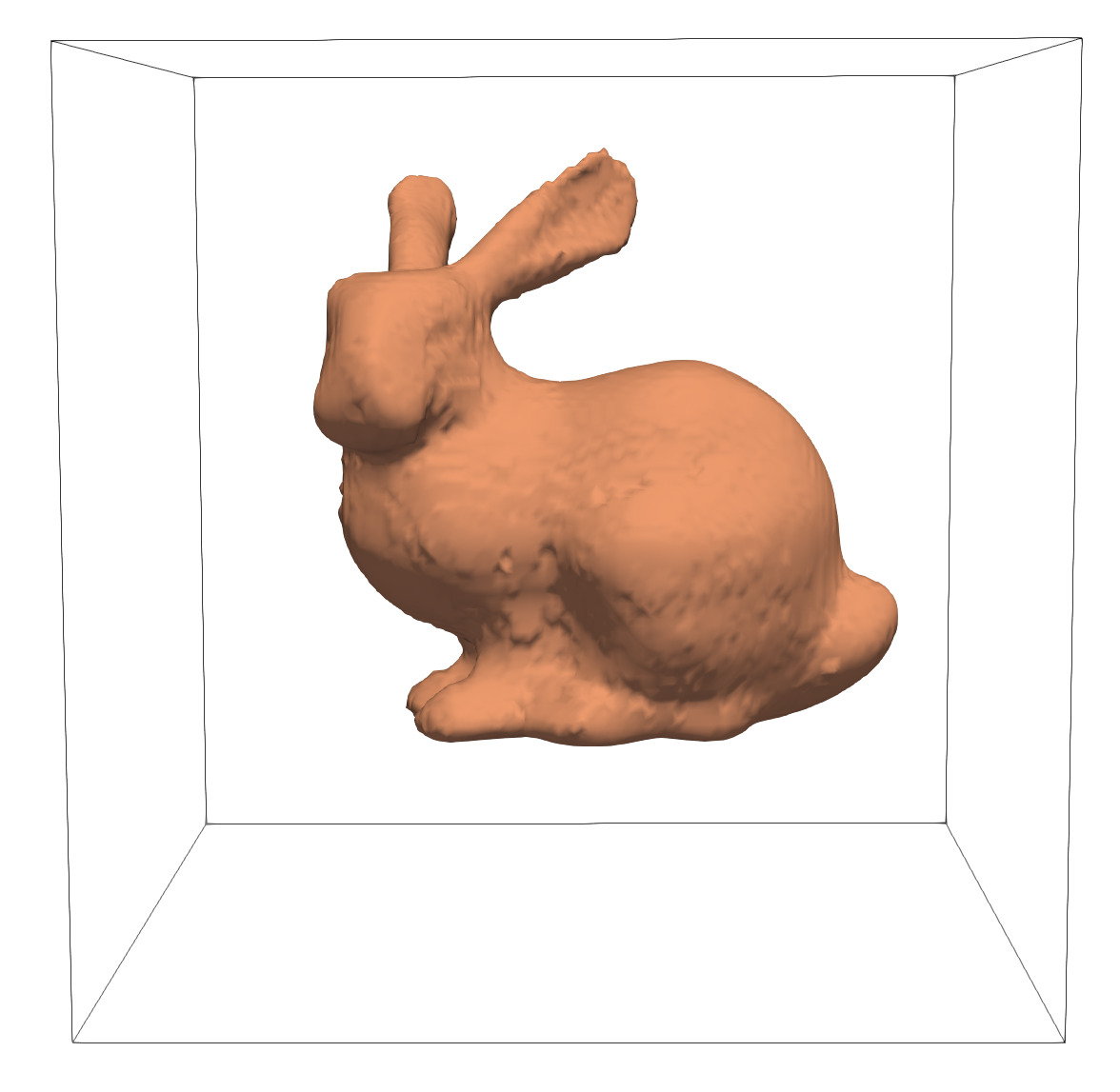} &
    \includegraphics[width=0.509\textwidth]{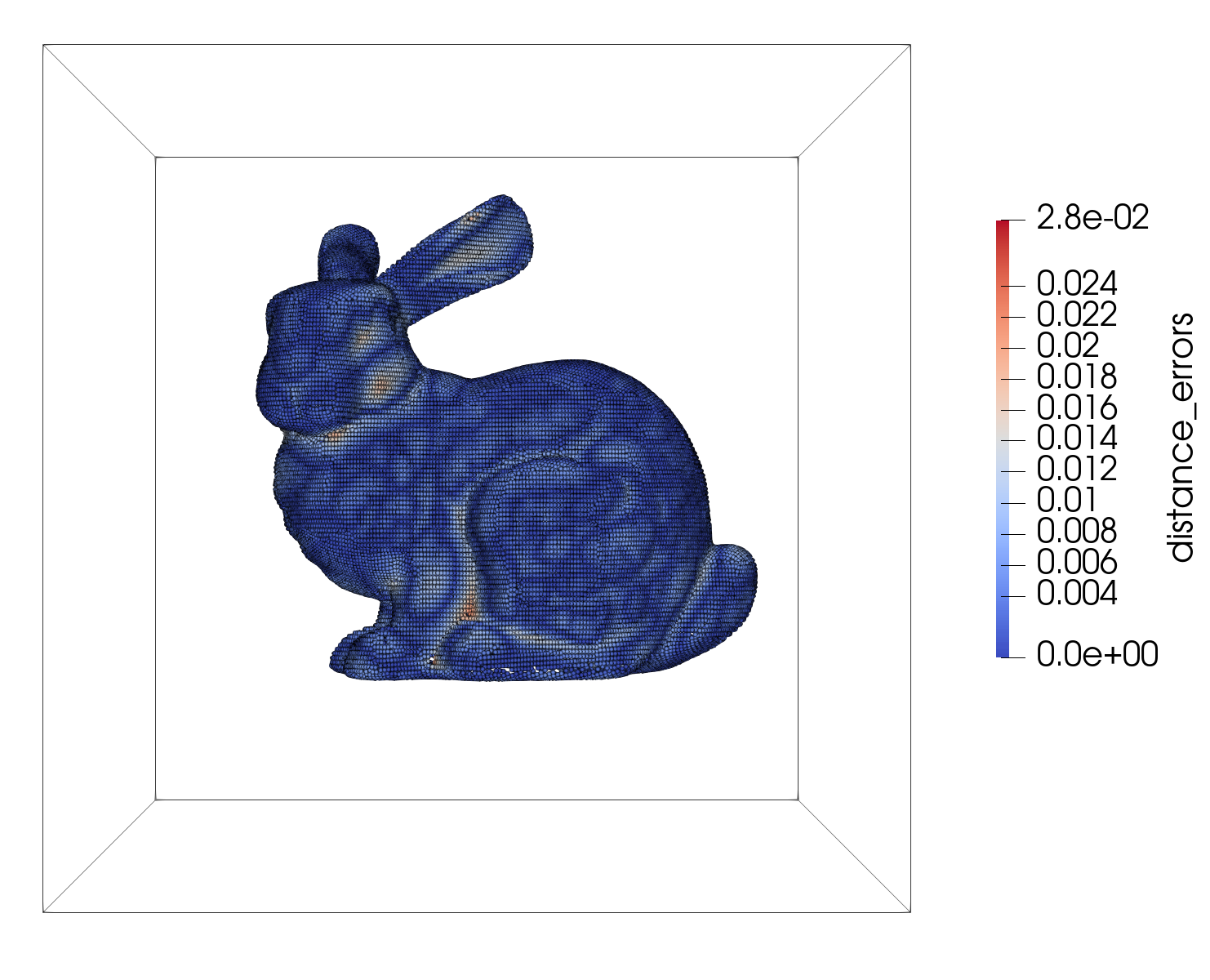}\\
   {\sc ParaView}'s iso-surface of the GPLS & \emph{Stanford Bunny}\footnote{available from http://graphics.stanford.edu/data/3Dscanrep/} with GPLS distance error
  \end{tabular}
  \caption{The GPLS approximation of the Stanford Bunny dataset. Left: visualisation of the level set $Q_{M}^{-1}(0)$ for $l_2$-degree $n=9$ derived from 4000 surface points. Right: The entire dataset with all $35.947$ surface points with color corresponding to the closest point distance to the GPLS. % The sub-sampled 4000 of those used to derive the GPLS are in blue (same as in left panel).
  \label{fig:bunny_fitting}}
  \end{figure}

\begin{experiment}[Non-Algebraic surface]\label{exp:Bunny} We repeat Experiment~\ref{exp:REC} for the Stanford Bunny $S_B$ and the Spot surface $S_C$, for which Corollary~\ref{cor:Level}$(i)$ does not apply. Therefore, we sub-sample 4000 points $P$ and their normals $\eta(q)$, $q \in P$, uniform at random.
By moving the points along the dataset normals $q' =  q + \lambda \eta(q)$, $\lambda = \pm 0.005, \pm 0.01, \pm 0.035$, we generate a surrounding narrow band
with (relaxed) signed distance function $d(q') = \lambda$.
The GPLS
  $Q_M$ is derived by fitting $d$ according to section~\ref{sec:SDF}.
The GPLS quality is measured by computing the shortest distances $\mathrm{dist}(q,M) \in \R^+$ across the entire dataset $q \in P$ as in Experiment~\ref{exp:SF}. The maximum and mean errors (distances) $E_\infty$ / $E_{\mathrm{mean}}$  are listed in Tables~\ref{tbl:bunny_fit_errors},\ref{tbl:spot_fit_errors} for different choices of polynomial degree and $l_p$-degree.
\end{experiment}

For both datasets the lowest distance error (in bold) is reached for Euclidean $l_2$-degree, reflecting the discussion in Section~\ref{sec:App}, Remark~\ref{rem:FIT} and \cite{Lloyd2}
on the optimality of that choice. Fig.~\ref{fig:bunny_fitting}(left) and Fig.~\ref{fig:cow_fitting}(left) show the surface visualised from the most accurate GPLS using {\sc ParaView}'s iso-surface rendering. The colorbar plots in Fig.~\ref{fig:bunny_fitting}(right), Fig.~\ref{fig:cow_fitting}(left) indicate the distance errors of the GPLS to the original datasets, respectively.
%IFS: check for visualization induced errors!
%MH: Sachin can you check ?!
% The blue points are the 4000 surface points $P$ used for constructing the GPLS, showing that errors mostly stem from high-curvature regions (e.g., the tips of the ears and the paws). The panel on the right shows all points (i.e., all vertices of the surface triangulation) of the dataset with the same 4000 points $P$ used to compute the GPLS highlighted in blue.
%MH: Problem resolved ...may have to converge on the legends and font size
For the Stanford bunny the GPLS requires  $|C|=486$ polynomial coefficients, $C \in \R^{|A_{3,9,2}|}$, hence delivering a representation of $S_B$ with a compression ratio $r = 35.947 / 486 \approx 74$. For the Spot dataset $|C|=847$ polynomial coefficients are required, $C \in \R^{|A_{3,12,2}|}$, yielding  compression ratio  $r = 5856 / 847 =7$.
Regarding the results, we expect that the shown examples are at the limit of what the GPLS method can handle in terms of geometric complexity.

\begin{figure}[t!]
  \includegraphics[width=0.51\textwidth]{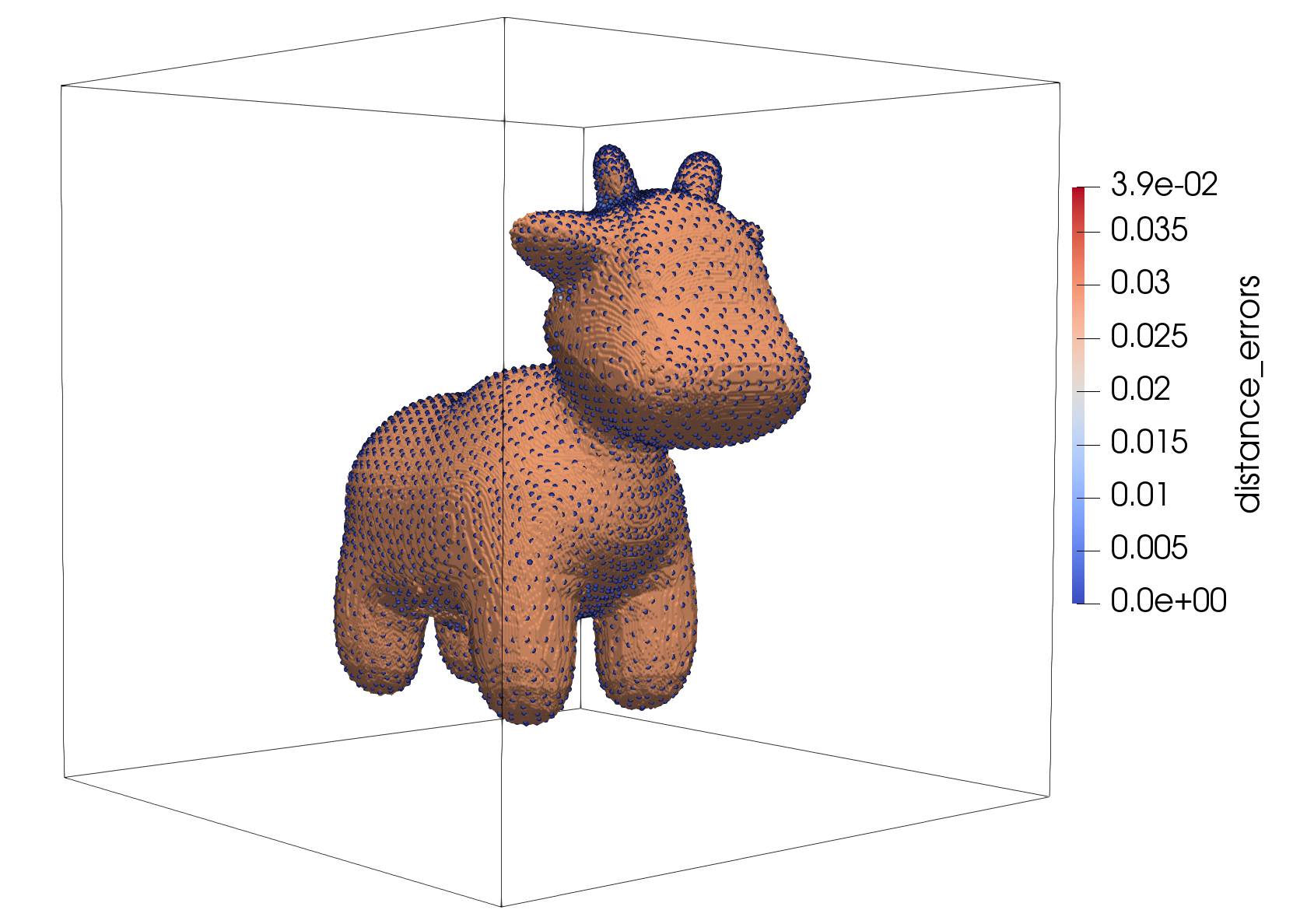}
  \includegraphics[width=0.44\textwidth]{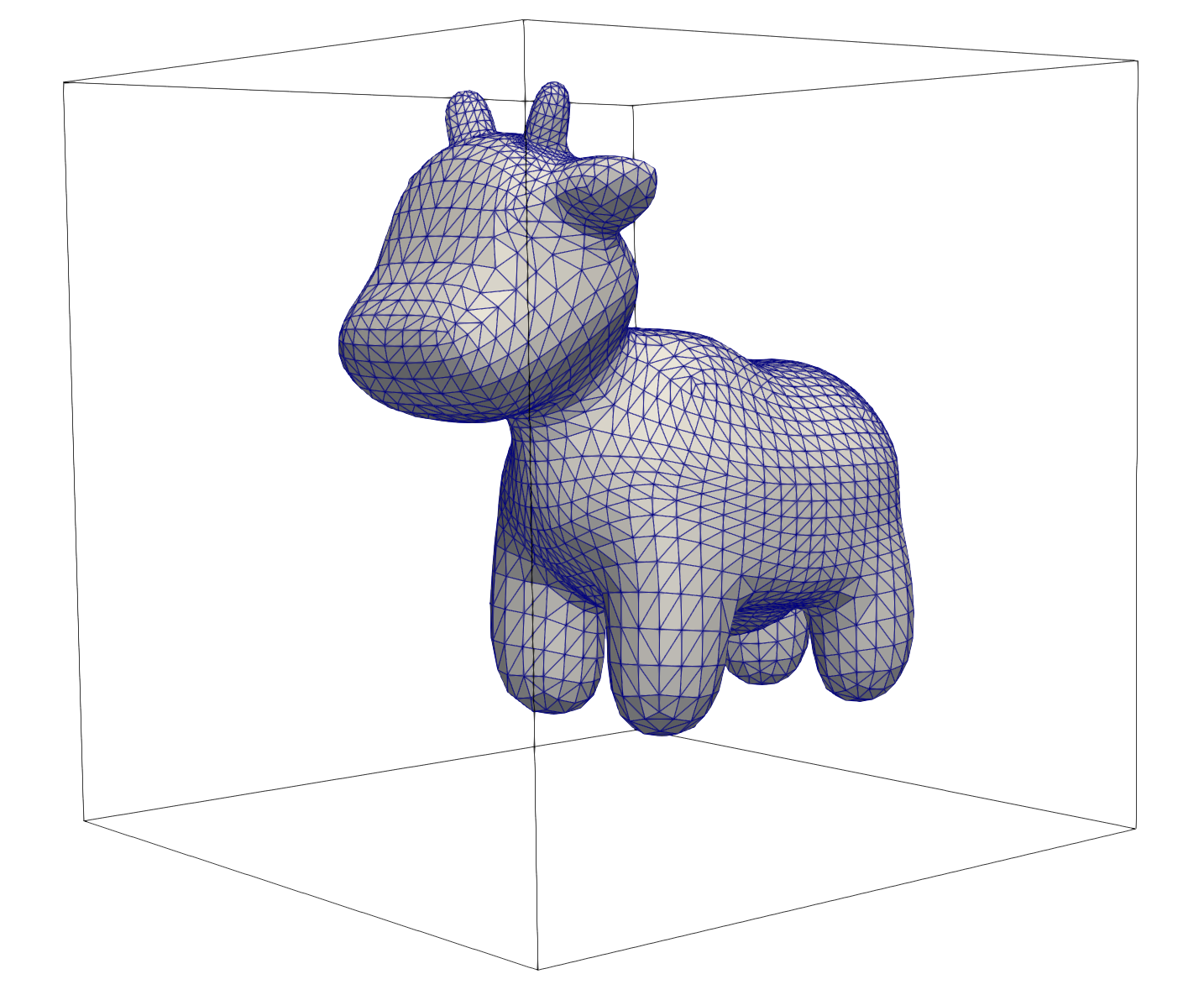}
   \centering
 \begin{tabular}{lccr}
 {\sc ParaView}'s iso-surface of the GPLS with
   & & &  \emph{Spot} dataset\footnote{available from https://www.cs.cmu.edu/~kmcrane/Projects/ModelRepository/}\\
  distance errors to the Spot triangle centres  & & &
\end{tabular}
\caption{The GPLS approximation of the Spot dataset. Left: {\sc ParaView}'s isocontour visualization of the level set $Q_{M}^{-1}(0)$ for $l_2$-degree $n=12$ derived from 4000 surface points with color corresponding to the closest point distance to the original Spot dataset (visualised as points). Right: The entire Spot dataset with all 5856 triangles. Triangle centres and their normals are used as data inputs for the GPLS.
\label{fig:cow_fitting}}
\end{figure}

\begin{table}[ht]
\centering
\begin{tabular}{lrrr}
degree $n$ & $E_\infty$ / $E_{\mathrm{mean}}$, $p=1$  &  $E_\infty$ / $E_{\mathrm{mean}}$, $p=2$ & $E_\infty$ / $E_{\mathrm{mean}}$, $p=\infty$   \\
\hline
% $1$ & $1.33$  & $4$   & $1.33$  & $4$ & $1.39$  & $8$ \\
% $2$ & $1.80$      & $10$  & $1.75$    & $11$  & $2.24$  & $27$ \\
% $3$ & $0.49$  & $20$  & $5.57$  & $29$  & $2.26$  & $64$ \\
% $4$ & $1.61$  & $35$  & $1.50$    & $54$  & $2.06$  & $125$ \\
% $5$ & $2.09$  & $56$  & $1.03$  & $99$  & $7.32$  & $216$ \\
% $6$ & $1.93$ /  & $84$  & $1.54$  & $163$ & $0.34$  & $343$ \\
$7$ & $0.418$ / $0.006$   & $0.075$ / $0.004$ & $0.035$ / $0.003$  \\
$8$ & $0.143$ / $0.005$   & $0.054$ / $0.003$ & $0.055$ / $0.003$\\
$\textbf{9}$  & $0.102$ / $0.004$   & $\textbf{0.029}$ / $\textbf{0.003}$ & $0.070$ / $0.002$  \\
$10$   & $0.043$ / $0.004$    & $0.049$ / $0.003$ &  $ 0.167$ / $0.002$ \\
$11$ & $0.068$ / $0.003$     & $0.022$ / $0.001$    & $0.160$ / $0.002$  \\
% $12$ & $0.15$ & $455$ & $0.17$  & $1069$  & $0.15$  & $2197$ \\
% $13$ & $0.16$ & $560$ & $0.20$  & $1355$  & $0.40$  & $2744$ \\
\end{tabular}
%IFS: I assume these are absolute errors? Maybe relative errors normalized to the domain edge length would be more meaningful? Or has the bunny been rescaled to live in the same unit cube as the algebraic test surfaces from the previous section? Otherwise it is hard to compare the errors...
%MH: Everything was normalized to hypercube and the error is really the closest point distance.
%IFS: well... if everything is normalized to a hypercube then how can an error be 2.09... or 7.32 ???
%MH: Sachin can you answer that ?!
%Sachin: The errors are now reported as the relative error of closest point distance to the surface.
%MH: Problem now resolved see Experiment 1 red marked text.

\caption{Maximum fitting (distance) errors for GPLS approximations of the Stanford Bunny surface with various polynomial degrees $n$ and $l_p$-degrees $p$. The best fit is highlighted in bold.}
\label{tbl:bunny_fit_errors}
\end{table}

\begin{table}[ht]
\centering
\begin{tabular}{lrrr}
degree $n$ & $E_\infty$ / $E_{\mathrm{mean}}$, $p=1$  &  $E_\infty$ / $E_{\mathrm{mean}}$, $p=2$ & $E_\infty$ / $E_{\mathrm{mean}}$, $p=\infty$   \\
\hline
$8$   & $0.099$ / $0.005$ & $0.110$ / $0.0030$ & $0.055$ / $0.0020$  \\
$9$   & $0.182$ / $0.004$ & $0.068$ / $0.0010$ & $0.123$ / $0.0005$  \\
$10$  & $0.082$ / $0.004$ & $0.063$ / $0.0010$   & $0.187$ / $0.0020$   \\
$11$  & $0.085$ / $0.003$ & $0.042$ / $0.0005$  & $ 0.036$ / $0.0003$   \\
$\textbf{12}$ &   $0.081$ / $0.001$    & $\textbf{0.029}$ / $\textbf{0.0003}$   & $0.044$ / $0.0010$   \\
\end{tabular}
\caption{Maximum fitting (distance) errors for GPLS approximations of the Spot surface with various polynomial degrees $n$ and $l_p$-degrees $p$. The best fit is highlighted in bold.}
\label{tbl:spot_fit_errors}
\end{table}

\begin{figure}[t]
\centering
\includegraphics[width=1.0\textwidth]{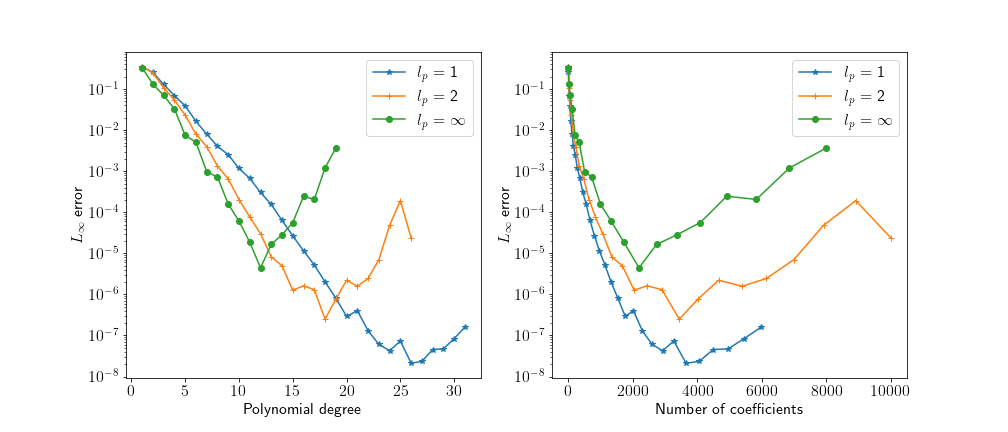}
\caption{Maximum approximation errors for fitting the Runge function on the Stanford Bunny from 10,000 uniformly randomly placed surface samples for different polynomial degrees. The total numbers of coefficients are plotted on the right.
\label{fig:bunny_runge_fitting}}
\end{figure}

Given the GPLS approximation of the Stanford Bunny surface $S_B$, we next address the task of globally fitting a scalar function $f: S_B \lo \R$ on the surface.

\begin{experiment}[Function fitting on non-algebraic surface]  We
sample the Runge function $f(x) = 1/(1 + |x|^2)$
at
10,000 randomly chosen surface points $P \subseteq S_B$ on the Stanford Bunny and apply the multivariate regression scheme from Remark~\ref{rem:fit} to derive approximations $Q_{P, f, A_{3,n,p}} \approx f$ of $f$ on $S_B$.
The $L_\infty$ approximation errors are measured across $500$ random surface points not used for the regression and plotted in Fig.~\ref{fig:bunny_runge_fitting} as a function of $p$ and $n$.
\label{exp:fct_fit}
\end{experiment}

All regressions achieve reasonable approximation of
the Runge function. Regression with respect to Euclidean
and maximum degree ($p=2,\infty$) convergences faster with degree $n$ than total-degree regression ($p=1$), confirming the expectations of section~\ref{sec:App}.
However, $l_1$ regression reaches the overall best approximation. In contrast, $l_{2,\infty}$ regression becomes unstable (for $n>11$ or $n> 17$, respectively) with maximum degree $p=\infty$ performing worst. Euclidean regression ($p=2$) reaches a $10^{-7}$ approximation fastest (for degree $n =17$), but the specific sample point distribution used here hampers its optimality in terms of coefficient count, as formulated in Corollary~\ref{cor:uni1}.
An extended discussion of these effects is provided by \cite{platte:2011,Lloyd2}, including an explanation for the observed numerical instabilities.

Since the Runge function is highly varying and notoriously hard to interpolate (``Runge's Phenomenon''), the accuracies reached here suggest that a larger class of functions can be approximated using the present method,
supporting classic computational tasks in differential geometry.

\section{Conclusion}\label{sec:con}

We have combined basic algebraic geometry and classic numeric \linebreak
analysis to approximate
smooth closed surfaces $S\subset \R^3$ by algebraic varieties with global polynomial level set (GPLS) $M=Q_M^{-1}(0)\approx S$. We proved uniqueness of these approximations in  Theorem~\ref{theorem:Dual}, with further discussion given in Remark~\ref{rem:unique}.
We presented
numerical experiments of computing differential-geometric quantities (curvatures and Laplacian of curvature) of algebraic surfaces approximated by their GPLS. Both the computational efficiency, in terms of the surface point counts $|P|$, $P\subseteq S$, as well as the accuracy reached by the GPLS method were superior to Curved Finite Elements (CFE) and to Closest-Point Finite Differences (CP-FD) by orders of magnitude.

We then estimated the limitations of GPLS methods in terms of the reachable surface complexity in Theorem~\ref{theo:APP} and numerically demonstrated them in the example of the \emph{Stanford Bunny} and the \emph{Spot dataset}. We then achieved global approximation of the highly varying Runge function on the surface of the Stanford Bunny, suggesting that the presented approach applies to a larger class of surfaces and functions, as for example
occurring in \emph{biophysics} \cite{mietke2019minimal,seifert1997configurations,sbalzarini2006simulations,colin2022quantifying} or mechanics \cite{schwartz1998simulation,sander2016numerical}.

In the present work, we focused on static surfaces. Our results, however, suggest that the proposed method could also provide a starting point
for dynamic surface deformation simulations, potentially providing an alternative to well-established \emph{level set methods and fast marching methods} \cite{sethian1999level}.

We also note that the concept of GPLS is not limited to two-dimensional surfaces, but can be extended to higher-dimensional embedded
(hypersurfaces) manifolds $\mathcal{M}\subseteq \R^m$, $m \in \N$. There, the computational  efficiency of the GPLS approach in terms of the required number of points $P\subseteq \mathcal{M}$ may pave the way for realising numerical the manifold models required, for instance, for Ricci-DeTurck flow  simulations \cite{fritz2015numerical}.

\section*{Acknowledgments}
We are deeply grateful for the insights and support we received  in discussions with
Prof.~Oliver Sander (TU Dresden). We want to thank Dan Fortunato (CCM Simons Foundation) for the Spot dataset and fruitful discussions on the subject we had.

% This work was partially funded by the Center of Advanced Systems Understanding (CASUS), financed by Germany's Federal Ministry of Education and Research (BMBF) and by the Saxon Ministry for Science, Culture and Tourism (SMWK) with tax funds on the basis of the budget approved by the Saxon State Parliament.

\bibliographystyle{siamplain}
\bibliography{/Users/admin_hecht93/Documents/REF/Ref.bib}
\end{document}